\documentclass[a4paper, 10pt]{article}
\usepackage{srcltx,graphicx,epstopdf}
\usepackage{amsmath, amssymb, amsthm}
\usepackage{color}
\usepackage{xcolor}
\usepackage{multirow}
\usepackage{subfig} 
\usepackage{hyperref}
\usepackage[margin=1in]{geometry}
\usepackage{cleveref}
\crefname{section}{Section}{Sections}
\crefname{subsection}{Subsection}{Subsections}
\crefname{appendix}{Appendix}{Appendix}
\crefname{figure}{Figure}{Figures}
\crefname{table}{Table}{Tables}
\crefname{property}{Property}{Properties}
\crefname{theorem}{Theorem}{Theorem}
\usepackage{booktabs}
\usepackage{threeparttable}
\usepackage[normalem]{ulem}
\graphicspath{{images/}}

\newtheorem{theorem}{Theorem}

\newtheorem{proposition}{Proposition}

{\theoremstyle{remark} \newtheorem{remark}{Remark}}

\newcommand\pd[2]{\dfrac{\partial {#1}}{\partial {#2}}}

\newcommand\bA{{\bf A}}
\newcommand\bbA{\bar{\bf A}}

\newcommand\bbI{\bar{{\bf I}}}

\newcommand\bB{{\bf B}}
\newcommand\mE{\mathcal{E}}

\newcommand\bx{\boldsymbol{x}}
\newcommand\bn{\boldsymbol{n}}

\newcommand\bc{\boldsymbol{c}}

\newcommand{\imag}{\mathrm{i}}

\newcommand\bbN{\mathbb{N}}
\newcommand\bbZ{\mathbb{Z}}
\newcommand\bbS{\mathbb{S}}

\newcommand\dd{\,\mathrm{d}}

\newcommand\mF{\mathcal{F}}
\newcommand\mG{\mathcal{G}}
\newcommand\mS{\mathcal{S}}

\def\bx{\boldsymbol{x}}

\def\bn{\boldsymbol{n}}

\def\bsOmega{\boldsymbol{\Omega}}

\numberwithin{equation}{section}

\definecolor{electricpurple}{rgb}{0.75,0.0,1.0}
\definecolor{darkred}{rgb}{0.65,0,0}
\definecolor{green}{rgb}{0.0, 0.5, 0.0}

 \newcommand\deletei{\bgroup\markoverwith{\textcolor{darkred}{\rule[0.5ex]{1pt}{1pt}}}\ULon}
 \newcommand\deleteii{\bgroup\markoverwith{\textcolor{blue}{\rule[0.5ex]{2pt}{2pt}}}\ULon}

\title {An asymptotic-preserving IMEX method for
  nonlinear radiative transfer equation}

\author{Jinxue
  Fu\thanks{Beijing Computational Science Research Center, Beijing,
    China, 100193, email: {\tt jinxue.fu@csrc.ac.cn}},
    ~~Weiming Li\thanks{Institute of Applied Physics and Computational Mathematics,
    Beijing, China, 100094, email: \tt{liweiming@pku.edu.cn}},~~Peng
  Song\thanks{Laboratory of Computational Physics, Institute of
    Applied Physics and Computational Mathematics, Beijing, China,
    100088, email: \tt{song\_peng@iapcm.ac.cn}}, ~~Yanli
  Wang\thanks{Beijing Computational Science Research Center, Beijing,
    China, 100193, email: {\tt ylwang@csrc.ac.cn}.} }

\begin{document}
\maketitle

\begin{abstract}
  We present an asymptotic preserving method for the radiative
  transfer equations in the framework of $P_N$ method.  An implicit
  and explicit numerical scheme is proposed to solve the $P_N$ system
  based on the order analysis of the expansion coefficients of the
  specific intensity, where the order of each expansion coefficient is
  derived by the Chapman-Enskog method. The coefficients at
  higher-order are treated explicitly while those at lower-order are
  treated implicitly in each equation of the $P_N$ system.  Energy
  inequality is proved for this numerical scheme. Several numerical
  examples validate the efficiency of this scheme in both optically
  thick and thin regions.

  \vspace*{4mm}
  \noindent {\bf key word:} radiative transfer equation; asymptotic
  preserving; energy stability

\end{abstract}

\section{Introduction}
Radiation plays an important role in thermal radiative transfer in
inertial confinement fusion. Thermal radiative transfer is an
intrinsic component of coupled radiation-hydrodynamic problems
\cite{semi2008Ryan}, and the radiative transfer equations (RTE) are
adopted to describe the energy exchange between different materials in
the system. However, this system is of high dimensionality. Generally,
there are all together seven independent variables in the system, such
as the position in the physical space, angle in the phase space, frequency
and time, which will lead to high computational cost
\cite{implicit2016Laboure}. At the same time, the radiation travels at
the speed of light, which imposes a quite restrictive limit on the time-step
size. Solving the RTE system numerically is a challenging problem
\cite{sun2015asymptotic1}.

Generally speaking, there are two kinds of methods to solve this
system, the stochastic and deterministic methods. One of the popular
stochastic methods is the implicit Monte Carlo (IMC) method
\cite{fleck1971}, which is efficient in optically thin regions, but
needs quite a large amount of particles in the optically thick
regions, making it quite expensive \cite{fleck1971}. Moreover, though
there is no ray effect in the stochastic method, it suffers from the
statistical noise which will also make this method inefficient
\cite{fleck1971}.  Several efforts have been made to improve the
efficiency of IMC method, such as \cite{Gentile2001, dens2011,
  dens2015, Park2012}, which we will not discuss in detail
here. Recently, a series of unified gas-kinetic schemes (unified
gas-kinetic particle method (UGKP) \cite{shi2020} and unified
gas-kinetic wave particle method (UGKWP) \cite{liweiming2020}) are
proposed to solve RTE system, where a particle-based Monte Carlo
solver is proposed to track the non-equilibrium transport.  For the
deterministic methods, the discrete-ordinates ($S_N$) method is often
adopted \cite{koch1995, Lath1965}. In this method, the transport
equation is solved along particular directions and the energy density
is reconstructed using a quadrature rule. $S_N$ method has been
studied for many years and several efforts have been made to improve
the efficiency of this method \cite{Warsa2004}. However, $S_N$ methods
suffer from ray-effects \cite{Morel2003, Math1999}, which will lead to
the phenomenon of hot spot in the simulation.

Another deterministic method is the spherical harmonics ($P_N$) method
\cite{Kershaw1976, Lewis1993}. In the framework of $P_N$ method, the
specific intensity of radiation is approximated by a series expansion
of polynomials in the angular space. $P_N$ method, which is one of the
spectral methods, may have high approximate efficiency, and also
preserves the property of rational invariance. However, for the cases
that the interactions with the material are rare, $P_N$ method may
lead to non-physical oscillations or even negative energy density
solution \cite{McCl2010}. Several attempts are made to correct the
negativity in $P_N$ equations, such as adding artificial scattering
terms \cite{Olson2009} or adding filter which is also known as the
filtered $P_N$ method \cite{mcclarren2010robust, implicit2016Laboure}.
Besides, for both $P_N$ and $S_N$ methods, due to the fact that
photons transport at an extremely fast speed, we usually have to treat
the transport term implicitly when employing time
discretization. Moreover, in the optically thick regime, the photon's
mean free path is quite small. Thus, the spatial mesh size, which
should be comparable to the photon's mean free path, is also very
small and will lead to very expensive computational cost
\cite{sun2015asymptotic1}.

The asymptotic preserving (AP) scheme for the kinetic equation solves
this problem by capturing the asymptotic limit of the kinetic equation
on the discrete level without the need to resolve small scales
\cite{Jin1991, Jin1993, Jin2000}. A scheme is called an AP scheme if
its asymptotic limit as the mean free path goes to zero with the time
step and mesh size fixed becomes a consistent and stable
discretization of the limit macroscopic equation (for the radiative
transfer equation, the limiting equation is a diffusion equation)
\cite{Klar1998, Larsen1987, Larsen1989, Mie2013}.  In the simulation
of the steady neutron transport problems, where the AP schemes were
first studied, some work has been done, such as those by Larsen,
etc. \cite{Larsen1987, Larsen1989} and Jin, etc. \cite{Jin1991,
  Jin1993}. Then, the AP schemes were later applied to the unsteady
problems, where several kinds of AP schemes were developed. In
\cite{Lemou2010}, the micro-macro decomposition is utilized to split
the distribution function, and the implicit-explicit (IMEX) scheme is
applied for the time discretization, where the discontinuous Galerkin
discretization is adopted in the spatial space \cite{xiong2020,
  xiong2015, Peng2020}, and the finite difference discretization is
utilized in \cite{Laiu2019}. In \cite{sun2015asymptotic1,
  sun2018asymptotic}, the UGKS method with AP property is developed
for the radiative transfer equations, where a linearized iterative
solver for the temperature is utilized. In \cite{tang2021}, the
three-state update is adopted to capture the correct front propagation
in the diffusion limit. Moreover, the Eulerian method for the
equilibrium part combined with a Monte Carlo solver for the
perturbation was proposed in \cite{crestetto2019}. In
\cite{Hammer2019}, the multiscale high/low order (HOLO) method is
utilized to build the AP scheme \cite{Maginot2016}, where the
higher-order S-stable diagonally implicit RK method with the
linearization of the Planck function is applied. In this paper, we
will develop an AP scheme for the gray approximation to the radiation
transfer equations in the framework of $P_N$ method. The specific
intensity is first approximated by a series expansion of the basis
functions. Then, the Chapman-Enskog expansion is utilized to get the
order of the expansion coefficients with respect to a parameter
$\epsilon$, which is the typical mean free path divided by the
macroscopic length scale, based on which an implicit-explicit scheme
is designed for $P_N$ system. In this scheme, the terms at
higher-order of $\epsilon$ are solved explicitly with those at
lower-order solved implicitly in each equation of the $P_N$ system. In
this case, the implicit-explicit $P_N$ system is changed into a pseudo
implicit system, which could be solved at the computational cost of an
explicit scheme.  Moreover, the energy exchange term is solved
implicitly, which will greatly release the restriction on the time
step length. The equation for the material energy is solved coupled
with $P_N$ system, which is reduced into a fourth degree polynomial
equation.
  
The numerical properties of the new scheme are also studied in this
work, including the stability property and the AP property.  The
stability properties of the numerical scheme are studied by the
Fourier analysis and the energy stability analysis. As to the AP
property, when the parameter $\epsilon$ goes to zero, the resulting
$P_N$ system is reduced into a finite difference scheme for the
material temperature. Numerical examples are tested first to validate
the AP property of this numerical scheme. The classical Marshak wave
problems in 1D spatial space and the lattice problem and the hohlraum
problem in 2D spatial space are tested to verify the efficiency of
this numerical scheme.

The rest of this paper is organized as follows: Section \ref{sec:pre}
will introduce the RTE system and $P_N$ method.  The AP IMEX method is
presented and discussed in detail in Section \ref{sec:method}, with
the AP property and energy stability proved in Section
\ref{sec:proof}.  Several numerical examples will be exhibited in
Section \ref{sec:num}. The conclusion and future work will be stated
in Section \ref{sec:conclusion}. $P_N$ system for the 1D RTE system and
the boundary conditions are discussed in Appendix \ref{app:1D} and
\ref{app:boundary_condition}, respectively. The Fourier analysis of
the numerical scheme and the proof of energy stability are discussed
in Appendix \ref{app:fourier} and \ref{app:energy_proof}, respectively.


\section{Radiative transfer equations and $P_N$ method}
\label{sec:pre}
In the absence of hydrodynamic motion and heat conduction, the
radiative transfer equations (RTE) are composed by a transport
equation of the specific intensity and the associated energy balance
equation. In this section, we will introduce the gray approximation to
the radiative transfer equations and the $P_N$ method, which is one of
the most popular numerical methods to solve RTE.

\subsection{The gray approximation to radiative transfer equations}
The radiative transfer and the energy exchange between radiation and
material are described by the gray approximation to the radiative
transfer equations, which have the form below:
\begin{subequations}
  \label{eq:RTE}
  \begin{align}
    \label{eq:I_rt}
    & \dfrac{\epsilon^2}{c} \pd{I}{t} + \epsilon \bsOmega \cdot \nabla I
    = \sigma \left(\frac{1}{4\pi}acT^4 - I\right), \\
    \label{eq:T}
   &  \epsilon^2 C_{v} \pd{T}{t} \equiv \epsilon^2 \pd{U}{t} =
    \sigma\left(\int_{\bbS^2}  I  \dd \bsOmega - ac T^4\right).
      \end{align}
\end{subequations}
Here $I(\bx, t, \bsOmega)$ is the specific intensity of radiation.
$\bsOmega$ is the angular variable which lies on $\bbS^2$, the surface
of the unit sphere. $\bx = (x, y, z)$ is the spatial variable, and
$\sigma(\bx, T)$ is the opacity. $\epsilon$ is the ratio between the
typical mean free path and the macroscopic length scale
\cite{Mie2013}, which plays a similar role to the Knudsen number in
the rarefied gas dynamics. In \eqref{eq:RTE}, the external source and
scattering terms are omitted.  $T(\bx, t)$ is the material temperature
and $c$ is the speed of light. $a$ is the radiation constant given by
\begin{equation}
  \label{eq:coe_a}
  a =  \frac{8 \pi k^4}{15 h^3 c^3},
\end{equation}
where $h$ is Planck's constant while $k$ is Boltzmann constant.  

The relationship between the material temperature $T(\bx, t)$ and the
material energy density $U(\bx, t)$ is
\begin{equation}
  \label{eq:U_T}
  \pd{U}{T} = C_{v} > 0,
\end{equation}
where $C_{v}(\bx, t)$ is the heat capacity. Integrating
\eqref{eq:I_rt} against $\bsOmega$, and together with \eqref{eq:T}, we
can get the conservation of energy
\begin{equation}
  \label{eq:energy}
  \epsilon^2 C_{v} \pd{T}{t} + \epsilon^2 \pd{E}{t}
  + \epsilon
  \int_{\bbS^2} \bsOmega \cdot \nabla I  \dd
  \bsOmega = 0, 
\end{equation}
where $E$ is the energy density defined as
\begin{equation}
  \label{eq:energy_I_1}
  E = \frac{1}{c}\int_{\bbS^2} I \dd \bsOmega. 
\end{equation}
The total energy  is then defined as
\begin{equation}
  \label{eq:total_energy}
  \mathcal{E} = U + E. 
\end{equation}
When $\epsilon$ goes to zero, the specific intensity $I$ goes to a
Planckian at the local temperature \cite{sun2018asymptotic,
  sun2015asymptotic1}, and the corresponding local temperature
$T^{(0)}$ satisfies the nonlinear diffusion equation
\begin{equation}
  \label{eq:limit}
  \pd{U(T^{(0)})}{t} + a \pd{}{t}\left(T^{(0)}\right)^4 = \nabla \cdot \frac{a
    c}{3 \sigma} \nabla \left(T^{(0)}\right)^4, \qquad I^{(0)} = ac\left(T^{(0)}\right)^4.
\end{equation}
In this approximation, the radiative flux $F(t, \bx)$ is related to
the material temperature by the Fick's law of diffusion given by
\begin{equation}
  \label{eq:Fick}
 F(t, \bx) = \int_{\bbS^2} \bsOmega I \dd \bsOmega  = -\frac{ac}{3 \sigma} \nabla T^4.
\end{equation}
Moreover, at this time the total energy \eqref{eq:total_energy} is
expressed as
\begin{equation}
  \label{eq:total_energy_fin}
  \mE = U^{(0)} + a \left(T^{(0)}\right)^4.
\end{equation}
When there is no energy exchange between the specific intensity and the material,
we get another form of the radiative transfer equation
\cite{seibold2014starmap, Mie2013} as
\begin{equation}
  \label{eq:linear_RTE}
  \dfrac{\epsilon^2}{c} \pd{I}{t} + \epsilon \bsOmega \cdot \nabla I
  = \sigma \left(\frac{1}{4\pi} \int I \dd \bsOmega - I\right), \\
\end{equation}
where the linear operator $\frac{1}{4\pi} \int I \dd \bsOmega - I$
models the scattering of the particles by the medium \cite{Mie2013}.

Direct simulation of the radiative transfer equations \eqref{eq:RTE}
is costly for several reasons. First, the independent variables of
\eqref{eq:RTE} are position, angle, and time, which are usually
seven-dimensional, making it expensive to simulate.  Then, radiation
travels at the speed of light, which makes the limit on the time-step
length for time-explicit schemes quite restrictive. Finally, when the
parameter $\epsilon$ is small, \eqref{eq:RTE} contains stiff source
terms, leading to stringent restrictions on the time step for
time-explicit schemes as well. On the other hand, \eqref{eq:RTE} goes
to the diffusion limit \eqref{eq:limit} as $\epsilon$ approaches zero,
which requires us to construct the asymptotic-preserving (AP) schemes
to solve this problem \cite{Jin2000, Jin1993}.


\subsection{$P_N$ system}
The $P_N$ method approximates the angular dependence of \eqref{eq:RTE}
by a series expansion of the spherical harmonics function. The moments
of the specific intensity is defined as
\begin{equation}
  \label{eq:moment}
  I_{l}^{m}(t, \bx) = 2\sqrt{\pi} \int_{\bbS^2} \overline{Y}_{l}^m(\bsOmega) I(
  t, \bx, \bsOmega) \dd \bsOmega, 
\end{equation}
where $Y_{l}^{m}(\bsOmega)$ is the spherical harmonics
\begin{equation}
  \label{eq:sp}
  Y_{l}^{m}(\bsOmega) =  \sqrt{\frac{2l + 1}{4 \pi}
  \frac{(l-m)!}{(l+m)!}} P_l^m(\cos \theta) \exp(im\phi), \qquad
  \bsOmega = (\sin \theta \cos \phi, \sin \theta \sin \phi,
  \cos \theta)^T,
\end{equation}
with $P_l^m(x)$ an associated Legendre polynomial.  Multiplying
\eqref{eq:I_rt} on both sides by $\overline{Y^m_l}(\bsOmega)$ and
integrating over $\bsOmega$, we can derive the detailed form of the
$P_N$ equations for \eqref{eq:I_rt} 
  \begin{equation}
  \label{eq:moment_equation_3d}
  \begin{split}
    & \frac{\epsilon^2}{c} \pd{I_l^m}{t} +
    \frac{\epsilon}{2}\pd{}{x}\left( -C_{l-1}^{m-1}I_{l-1}^{m-1} +
      D_{l+1}^{m-1} I_{l+1}^{m-1}+ \mE_{l-1}^{m+1} I_{l-1}^{m+1} -
      F_{l+1}^{m+1} I_{l+1}^{m+1}
    \right)\\
    & \qquad + \frac{\imag \epsilon}{2}\pd{}{y}\left(
      C_{l-1}^{m-1}I_{l-1}^{m-1} - D_{l+1}^{m-1} I_{l+1}^{m-1}+
      \mE_{l-1}^{m+1} I_{l-1}^{m+1} - F_{l+1}^{m+1} I_{l+1}^{m+1}
    \right) \\
& \qquad + \epsilon\pd{}{z} \left(A_{l-1}^m I_{l-1}^m +
      B_{l+1}^mI_{l+1}^m\right) = -\sigma I_l^m + \sigma ac T^4
    \delta_{l0}\delta_{m0}, \qquad l \in \bbN, \qquad m \in \bbZ,
    \quad |m| \leqslant l.
  \end{split}
\end{equation}
Here, we only consider the problem which is symmetric with the $x-z$
plane, and then \eqref{eq:moment_equation_3d} is reduced into the 2D
equation \cite{McCl2008}
 \begin{equation}
  \label{eq:moment_equation}
  \begin{split}
    & \frac{\epsilon^2}{c} \pd{I_l^m}{t} +
    \frac{\epsilon}{2}\pd{}{x}\left( -C_{l-1}^{m-1}I_{l-1}^{m-1} +
      D_{l+1}^{m-1} I_{l+1}^{m-1}+ \mE_{l-1}^{m+1} I_{l-1}^{m+1} -
      F_{l+1}^{m+1} I_{l+1}^{m+1}
    \right)\\
    & \qquad + \epsilon\pd{}{z} \left(A_{l-1}^m I_{l-1}^m +
      B_{l+1}^mI_{l+1}^m\right) = -\sigma I_l^m + \sigma ac T^4
    \delta_{l0}\delta_{m0}, \qquad l \in \bbN, \qquad m \in \bbZ,
    \quad |m| \leqslant l,
  \end{split}
\end{equation}
where $\delta_{ij}$ is the Kronecker-delta function, and the
coefficients are
\begin{equation}
  \label{eq:coe_moment}
  \begin{split}
    & A_l^{m} = \sqrt{\frac{(l -m +1)(l + m +1)}{(2l+3)(2l+1)} }, \qquad
    B_l^m = \sqrt{\frac{(l-m)(l+m)}{(2l+1)(2l-1)}}, \\
    & C_l^{m} = \sqrt{\frac{(l +m +1)(l + m +2)}{(2l+3)(2l+1)} }, \qquad
    D_l^m = \sqrt{\frac{(l-m)(l-m-1)}{(2l+1)(2l-1)}}, \\
    & \mE_l^{m} = \sqrt{\frac{(l -m +1)(l - m +2)}{(2l+3)(2l+1)} }, \qquad
    F_l^m = \sqrt{\frac{(l+m)(l+m-1)}{(2l+1)(2l-1)}}. \\
  \end{split}
\end{equation}
The derivation of the moment system \eqref{eq:moment_equation} is
discussed in the literature, and we refer to
\cite{seibold2014starmap, McCl2008} for more details. Especially, the
governing equation for the zeroth moment $I_0^0$ is
\begin{equation}
  \label{eq:I0}
     \frac{\epsilon^2}{c} \pd{I_0^0}{t} + \frac{\epsilon}{2}\pd{}{x}
    \left( D_{1}^{-1} I_{1}^{-1}  - 
    F_{1}^1 I_{1}^{1} \right)+ \epsilon\pd{}{z} B_{1}^0I_{1}^0 = -\sigma
    I_0^0 + \sigma acT^4,
\end{equation}
and the energy density is
\begin{equation}
  \label{eq:relation_E_I}
  E = \frac{1}{c} \int_{\bbS^2} I(t, \bx, \bsOmega) \dd \bsOmega =
  \frac{I_0^0}{c} .
\end{equation}
In the framework of $P_N$ method, the governing equation of the total
energy \eqref{eq:energy} is
\begin{equation}
  \label{eq:Pn_energy}
  C_{v} \pd{T}{t} + \frac{1}{c}\pd{I_0^0}{t} +
  \frac{1}{\epsilon}
  \left(\frac{1}{2}\pd{}{x}
    \left(D_{1}^{-1} I_{1}^{-1}  - F_{1}^1 I_{1}^{1}\right)+ \pd{}{z} B_{1}^0I_{1}^0\right)
  = 0. 
\end{equation}
Moreover, a finite system is needed for the numerical simulation, and
the specific intensity $I(t, \bx,\bsOmega)$ is approximated as
\begin{equation}
  \label{eq:I}
  I(t,\bx, \bsOmega) \approx \sum\limits_{l \leqslant M} \sum\limits_{|m| \leqslant
    l} \frac{I_{l}^m}{2\sqrt{\pi}} Y_l^m(\bsOmega),  
\end{equation}
where $M$ is the truncation order.  Then, we can derive the final
$P_N$ system for (2.1) as 
\begin{equation}
  \label{eq:final_Pn}
  \begin{aligned}
&    \frac{\epsilon^2}{c} \pd{I_0^0}{t} + \frac{\epsilon}{2}\pd{}{x}
    \left( D_{1}^{-1} I_{1}^{-1}  - 
    F_{1}^1 I_{1}^{1} \right)+ \epsilon\pd{}{z} B_{1}^0I_{1}^0 = -\sigma
    I_0^0 + \sigma acT^4, \\
& \frac{\epsilon^2}{c} \pd{I_l^m}{t} +
    \frac{\epsilon}{2}\pd{}{x}\left( -C_{l-1}^{m-1}I_{l-1}^{m-1} +
      D_{l+1}^{m-1} I_{l+1}^{m-1}+ \mE_{l-1}^{m+1} I_{l-1}^{m+1} -
      F_{l+1}^{m+1} I_{l+1}^{m+1}
    \right)\\
    & \qquad + \epsilon\pd{}{z} \left(A_{l-1}^m I_{l-1}^m +
      B_{l+1}^mI_{l+1}^m\right) = -\sigma I_l^m, \qquad l  \leqslant M, \qquad m \in \bbZ,
    \quad |m| \leqslant l.
  \end{aligned}
\end{equation}
Here $I_{M+1}^m$ is simply set as zero to get the closed system as in
\cite{seibold2014starmap, McCl2008}, and the resulting $P_N$ system
is globally hyperbolic.  

Gathering \eqref{eq:final_Pn} and \eqref{eq:Pn_energy}, we obtain the
governing equations for the $P_N$ system, which will reduce the
computational complexity when simulating \eqref{eq:RTE}. It is widely
used to solve RTE, such as in \cite{semi2008Ryan, Mie2013,
  seibold2014starmap}. However, the time step limitation and the
multi-scale problem brought by the small mean free path still exist
for $P_N$ system. In the next sections, we will propose an AP scheme
for the $P_N$ system to release the restriction on the time step
length.


\section{Asymptotic-preserving IMEX method}
\label{sec:method}
In this section, we will introduce an asymptotic-preserving IMEX
numerical scheme to solve \eqref{eq:final_Pn} and
\eqref{eq:Pn_energy}. We will begin from the order analysis of the
expansion coefficients with respect to the parameter $\epsilon$, based
on which we will propose the new numerical scheme.

\subsection{Formal order analysis}

In this section, we will analyze the order of the expansion
coefficients $I_l^{m}$ based on $\epsilon$. One possible way to
describe the accuracy of the moment models in the near-continuum
region is through the Chapman-Enskog method. When in such a regime,
the parameter $\epsilon$ is regarded as a small number. Thus the
Chapman-Enskog expansion could be applied, and the specific intensity
$I$ is expanded in power series of $\epsilon$,
  \begin{equation}
    \label{eq:CE_I}
    I = I^{(0)} + \epsilon I^{(1)} + \epsilon^2 I^{(2)} +
    \cdots.
  \end{equation}
  Define
  \begin{equation}
    I_{l}^{m,(k)} =
    2\sqrt{\pi} \int_{\bbS^2} \overline{Y}_{l}^m(\bsOmega) I^{(k)}(
    t, \bx, \bsOmega) \dd \bsOmega.
  \end{equation}
  We claim that for
  fixed  $k$,
  \begin{equation}
    \label{eq:claim}
    I^{m,(k)}_l = 0, \qquad  {\rm for}\quad l > k. 
  \end{equation}
  This could be proved by mathematical induction through the following
  steps: \begin{enumerate} \item If $k$ equals zero, matching the
    order of $\mathcal{O}(1)$ in \eqref{eq:I_rt}
    shows
    \begin{equation}
      I^{(0)} = \dfrac{1}{4\pi} a c  T^4.
    \end{equation}
    Then, by the orthogonality of the spherical harmonic functions, it
    holds that
    \begin{equation}
      I^{m,(0)}_l = \int_{\bbS^2}\overline{Y}_l^m I^{(0)} \dd\bsOmega = 0, \quad \forall~
      l \geqslant 1.  
    \end{equation}
    Therefore, this claim holds for $k = 0$.

  \item Assuming this claim holds for $k \leqslant n$ as
    \begin{equation}
      \label{eq:k<n}
      I_l^{m, (k)} = 0, \qquad  \forall ~k \leqslant n, \qquad {\rm if}~ l > k,
    \end{equation}
    then, it holds that
    \begin{equation}
      \label{eq:k=n}
            I_l^{m,(n-1)} = 0, \quad I_l^{m, (n)} = 0,  \qquad {\rm if}~ l > n,
    \end{equation}
    and we consider the case $k = n+1$.

  \item If $k$ equals $n+1$, rewrite the governing equation
      \eqref{eq:moment_equation} into
      \begin{equation}
        \label{eq:moment_equation_1}
        \begin{split}
          & \frac{\epsilon^2}{c} \pd{I_l^m}{t} + \epsilon \nabla_{\bx}
          F(I_{l-1}^{m}) + \epsilon \nabla_{\bx} G(I_{l+1}^{m}) = -\sigma
          I_l^m, \qquad l > 0, \quad m \in \bbZ, ~ |m| \leqslant l,
        \end{split}
      \end{equation}
      where
      \begin{equation}
        \label{eq:coe}
        \begin{split}
          & \nabla_{\bx}F(I_{l-1}^{m}) = \frac{\partial}{2\partial x}\Big(
          -C_{l-1}^{m-1}I_{l-1}^{m-1} + \mE_{l-1}^{m+1} I_{l-1}^{m+1}
          \Big) + \pd{}{z} A_{l-1}^m I_{l-1}^m, \\
          & \nabla_{\bx}G(I_{l+1}^m) = \frac{\partial}{2\partial x}\Big(D_{l+1}^{m-1}
          I_{l+1}^{m-1}- F_{l+1}^{m+1} I_{l+1}^{m+1}\Big) + \pd{}{z}
          B_{l+1}^mI_{l+1}^m .
        \end{split}
      \end{equation}
      Matching the terms at order $\epsilon^{n+1}$ yields
\begin{equation}
  \begin{split}
    & \frac{1}{c} \pd{I_{l}^{m,(n-1)}}{t} + \nabla_{\bx}
    F(I_{l-1}^{m,(n)}) + \nabla_{\bx} G(I_{l+1}^{m,(n)}) = -\sigma
    I_l^{m,(n+1)}, \qquad l > 0, \quad m \in \bbZ, ~ |m| \leqslant l.
  \end{split}
\end{equation}
Therefore, we can derive that 
\begin{equation}
  \label{eq:I_l_m_ce}
  \begin{split}
    & I_l^{m,(n+1)} = -\dfrac{1}{\sigma} \left(\frac{1}{c}
      \pd{I_{l}^{m,(n-1)}}{t} + \nabla_{\bx} F(I_{l-1}^{m,(n)}) +
      \nabla_{\bx} G(I_{l+1}^{m,(n)})\right), \qquad l > 0, \quad m
    \in \bbZ, ~ |m| \leqslant l.
  \end{split}
\end{equation}
Due to the assumption \eqref{eq:k=n}, it holds that $I_l^{m, (n-1)}$,
$I_{l-1}^{m,(n)}$ and $I_{l+1}^{m,(n)}$ are all zero if $l >
n+1$. Therefore, $I_l^{m,(n+1)} = 0$ for $l > n+1$.

\item By induction, $I^{m,(k)}_l = 0$ for all $k$ when $l > k$, and we
  have proved the claim \eqref{eq:claim}.
\end{enumerate}

Then, based on \eqref{eq:claim}, we could derive the order of $I_l^m$. Precisely, 
from \eqref{eq:I0}, matching the order of $\mathcal{O}(1)$ shows
\begin{equation}
  \label{eq:order_I0}
  I_0^{0,(0)} = a c T^4.
\end{equation}
Thus, the leading order term of $I^0_0$ is $\mathcal{O}(1)$. Plugging
\eqref{eq:CE_I} into \eqref{eq:moment}, we can obtain the Chapman-Enskog
expansion of $I^m_l$ as 
\begin{equation}
  I^m_l = I^{m,(0)}_l + \epsilon I^{m,(1)}_l + \cdots + \epsilon^l
  I^{m,(l)}_l + \cdots.
\end{equation}
Based on \eqref{eq:claim}, we can conclude that the leading order of
$I_l^m$ is $\epsilon^l I_l^{m,(l)}$. Then, let $n$ equal $l-1$ in
\eqref{eq:I_l_m_ce}, it holds that
\begin{equation}
  \label{eq:order_Il}
  I_l^{m,(l)} = -\dfrac{1}{\sigma} \nabla_{\bx}
  F(I_{l-1}^{m,(l-1)}) = \mathcal{O}(1), \qquad l > 0, \quad m \in \bbZ, ~ |m| \leqslant l,
  \end{equation}
  with the other two terms equaling zero. Thus, the final expression of $I_l^m$ holds that 
  \begin{equation}
  \label{eq:order_l}
  I_l^m = -\epsilon^l \dfrac{1}{\sigma} \nabla_{\bx}
    F(I_{l-1}^{m,(l-1)}) + \mathcal{O}(\epsilon^{l+1}),
\end{equation}
which indicates  the order of $I_l^m$ is $\mathcal{O}(\epsilon^l)$ as
\begin{equation}
  \label{eq:order_Ilm}
  I_l^m = \mathcal{O}(\epsilon^l).
\end{equation}

\begin{remark} 
  Substituting \eqref{eq:order_I0} and
  \eqref{eq:order_Il} for $l=1$ into the equation of total energy
  \eqref{eq:Pn_energy}, we can obtain the same diffusion equation of $T$
  as \eqref{eq:limit} in the framework of $P_N$ method with $\epsilon$
  going to zero.
\end{remark}

\subsection{Semi-discrete IMEX methods}
For now, we have obtained the order of $I_l^m$ with respect to $\epsilon$,
based on which, we will introduce the semi-discrete scheme in time
with globally stiffly accurate IMEX RK scheme. From the formal order
analysis, it holds in \eqref{eq:moment_equation_1} that 
\begin{equation}
  \label{eq:order_Pn}
  \begin{split}
    \epsilon \nabla_{\bx} G(I_{l+1}^{m}) =
    \mathcal{O}(\epsilon^{l+2}), \qquad \epsilon \nabla_{\bx}
    F(I_{l-1}^{m}) = \mathcal{O}(\epsilon^l), \qquad -\sigma I_l^m +
    \sigma ac T^4 \delta_{l0}\delta_{m0} = \mathcal{O}(\epsilon^{l}).
  \end{split}
\end{equation}
The implicit-explicit strategy adopted here is to treat all the terms
with higher-order of $\epsilon$ explicitly while others implicitly.
Thus, in the numerical scheme,
$\nabla_{\bx} G(I_{l+1}^{m}) = \mathcal{O}(\epsilon^{l+2})$, which is
at the higher-order of $\epsilon$, is set as the explicit term while
$\nabla_{\bx} F(I_{l-1}^{m})$ and the energy exchange terms on the
right side, which are at the lower-order of $\epsilon$, are set as the
implicit terms.  Based on this, the first-order semi-discrete
scheme is proposed as below.

\paragraph{First-order scheme}
Given $(I_{l}^m)^n$ and $T^n$ to approximate the solution $I_l^m$ and
$T$ at time $t^n$, the first-order semi-discrete scheme to update the
specific intensity is
\begin{subequations}
  \label{eq:first_order}
  \begin{align}
    \label{eq:first_I0}
    & \frac{\epsilon^2}{c} \frac{(I_{0}^0)^{n+1} - (I_{0}^0)^n }{\Delta t}
      +   \epsilon\Big(\nabla_{\bx} G(I_{1}^{0}) \Big)^n =\sigma^n \Big(
      a c(T^{n+1})^4 - (I_{0}^0)^{n+1}\Big), \\
     \label{eq:first_T}
    & C_{v} \frac{T^{n+1} - T^n}{\Delta t} +
      \frac{1}{c}\frac{(I_0^0)^{n+1} - (I_{0}^0)^n}{\Delta t}  +
      \frac{1}{\epsilon}  
      \Big(\nabla_{\bx} G(I_{1}^{0}) \Big)^n  = 0,  \\
    \label{eq:first_I}
    & \frac{\epsilon^2}{c} \frac{(I_{l}^m)^{n+1} - (I_{l}^m)^n
      }{\Delta t} + \epsilon \Big( \nabla_{\bx} F(I_{l-1}^{m}) \Big)^{n+1} +
      \epsilon\Big(\nabla_{\bx} G(I_{l+1}^{m}) \Big)^n =
      -\sigma^{n+1}(I_{l}^m)^{n+1}. 
     \end{align}
\end{subequations}
Though the dominating terms in \eqref{eq:first_I} are treated
implicitly, this first-order numerical system \eqref{eq:first_order}
can be actually solved at the same computation cost as an explicit
scheme. This is for the reason that when solving \eqref{eq:first_I},
the terms with implicit scheme are already known.  In the
implementation, the coupled system \eqref{eq:first_I0}
and \eqref{eq:first_T} will be solved firstly to update $(I_0^0)^{n+1}$ and
$T^{n+1}$. Substituting \eqref{eq:first_I0} into \eqref{eq:first_T}
yields a fourth-order polynomial equation of $T^{n+1}$ as
\begin{equation}
  \label{eq:T_n+1}
  C_{v} T^{n+1} + \frac{\Delta t a c}{\epsilon^2+ \sigma^n \Delta
    t c}  (T^{n+1})^4 -\Bigg( C_{v} T^n + 
  \frac{\sigma^n \Delta t}{\epsilon^2 + \sigma \Delta t c}(I_0^0)^n -
  \frac{\Delta t^2 c \sigma^n}{\epsilon( \epsilon^2 + \sigma^n \Delta
    t c)} \Big(\nabla_{\bx} G(I_{1}^{m}) \Big)^n \Bigg)  =
    0,
\end{equation}
the solution of which will be guaranteed by the proposition below.
\begin{proposition}\label{lem:T_solution}
  The equation \eqref{eq:T_n+1} has only one positive solution if
  positive solutions for $I^0_0$ and $T$ exist.
\end{proposition}

\begin{proof}
 Let
 \begin{equation}
   \label{eq:f}
   f(T) =  C_{v} T + \frac{\Delta t a c}{\epsilon^2+ \sigma^n \Delta
    t c}  T^4 -\Bigg( C_{v} T^n + 
  \frac{\sigma^n \Delta t}{\epsilon^2 + \sigma \Delta t c}(I_0^0)^n -
  \frac{\Delta t^2 c \sigma^n}{\epsilon( \epsilon^2 + \sigma^n \Delta
    t c)} \Big(\nabla_{\bx} G(I_{1}^{m}) \Big)^n \Bigg).
\end{equation}
From \eqref{eq:f}, it is easy to verify that
 \begin{equation}
     \label{eq:f_T}
     f(0) < 0, \qquad f'(T) > 0.
 \end{equation}
 Then, this proposition holds.
 \end{proof}

\begin{remark}
Though $P_N$ is not a positive preserving method, for most of the
problems we tested, $I_0^0$ and $T$ are kept positive in the
computation. Positivity-preserving schemes for $I_0^0$ and $T$ will be
the subject of future investigation.
\end{remark}

  \begin{remark}
    To get $T^{n+1}$ efficiently, the GNC Scientific Library is
    utilized here to solve the fourth-order polynomial equation
    \eqref{eq:T_n+1}, which will make the computational time of the
    non-linear iteration to get $T^{n+1}$ negligible.
  \end{remark}

  This first-order semi-discrete numerical scheme can be extended to
  the higher-order IMEX RK scheme naturally, which is listed below.

\paragraph{Higher-order IMEX RK scheme}

To achieve higher-order accuracy in time, the globally stiffly
accurate IMEX RK scheme is adopted here. The IMEX RK scheme is widely
discussed \cite{Jang2015, xiong2015, Peng2020}. Thus, we only list the
scheme here. The higher-order scheme is combined with the same
implicit-explicit strategy as in the first-order case. Precisely, the
exact form is
\begin{subequations}
  \label{eq:high_order}
 \begin{align}
    \label{eq:high_order_I0}
    & \frac{\epsilon^2}{c} (I_{0}^0)^{n+1} = \frac{\epsilon^2}{c}
      (I_{0}^0)^n    - \Delta t \sum_{k = 1} ^{s} \tilde{b}_{k}
      \epsilon  \Big(\nabla_{\bx} G(I_{1}^{0}) \Big)^{n+1,k}  + \Delta t
      \sum_{k = 1} ^s b_{k}  \sigma^{n+1, k-1} \Big(a c (T^{n+1, k})^4 -
      (I_{0}^0)^{n+1, k}\Big), \\
    \label{eq:high_order_T}
     & C_{v} \frac{ T^{n+1} -  T^n}{\Delta t} +
       \frac{1}{c}\frac{(I_0^0)^{n+1} - (I_{0}^0)^n}{\Delta t}  
       + \frac{\Delta t}{\epsilon}  \sum_{k = 1} ^{s} \tilde{b}_{k}
       \left(\nabla_{\bx}  G(I_1^0)\right)^{n+1, k} = 0, \\
    \label{eq:high_order_I}
    & \frac{\epsilon^2}{c} (I_{l}^m)^{n+1} = \frac{\epsilon^2}{c} (I_{l}^m)^n 
      - \Delta t \sum_{k = 1} ^{s} \tilde{b}_{k} \epsilon
      \Big(\nabla_{\bx} G(I_{l+1}^{m}) \Big)^{n+1, k} \\\nonumber
    & \hspace{5cm} - \Delta t \sum_{k = 1} ^{s} b_{k}
      \Big( \Big( \epsilon\nabla_{\bx} F(I_{l-1}^{m})\Big)^{n+1,k}
      - \sigma^{n+1, k}(I_{l}^m)^{n+1, k}\Big) = 0,
  \end{align}
\end{subequations}
where the approximations at the internal stages of an RK step  satisfy
\begin{subequations}
      \label{eq:high_order_local}
      \begin{align}
    \label{eq:high_order_I0_local}
    & \frac{\epsilon^2}{c} (I_{0}^0)^{n+1, k} = \frac{\epsilon^2}{c}
      (I_{0}^0)^n    - \Delta t \sum_{j = 1} ^{k-1} \tilde{a}_{kj}
      \epsilon  \Big(\nabla_{\bx} G(I_{1}^{0}) \Big)^{n+1, j}  + \Delta t
      \sum_{j = 1} ^k a_{kj}  \sigma^{n+1, j-1} \Big(a c (T^{n+1, j})^4 -
      (I_{0}^0)^{n+1, j}\Big), \\
    \label{eq:high_order_T_local}
    & C_{v} \frac{ T^{n+1, k} -  T^n}{\Delta t} +
      \frac{1}{c}\frac{(I_0^0)^{n+1, k} - (I_{0}^0)^n}{\Delta t}  
      + \frac{\Delta t}{\epsilon}  \sum_{j = 1} ^{k-1} \tilde{a}_{kj}
      \left(\nabla_{\bx}  G(I_1^0)\right)^{n+1, j} = 0, \\
        \label{eq:high_order_I_local}
    & \frac{\epsilon^2}{c} (I_{l}^m)^{n+1, k} = \frac{\epsilon^2}{c} (I_{l}^m)^n 
      - \Delta t \sum_{j = 1} ^{k-1} \tilde{a}_{kj} \epsilon
      \Big(\nabla_{\bx} G(I_{l+1}^{m}) \Big)^{n+1, j} 
  \\ \nonumber 
  & \hspace{5cm} - \Delta t \sum_{j = 1} ^{k} a_{kj} \Big( \Big(\epsilon
      \nabla_{\bx} F(I_{l-1}^{m})\Big)^{n+1, j} 
      - \sigma^{n+1, j}(I_{l}^m)^{n+1, j}\Big) = 0.    
  \end{align}
\end{subequations}

    \begin{remark}
      The coefficients $I_{l}^m$ are also numerically solved
      successively as in \eqref{eq:high_order}. Since the lower-order
      terms in \eqref{eq:high_order_I} are already known, the
      convection terms can be derived explicitly. Moreover, the
      opacity $\sigma^{n+1, l}$ in \eqref{eq:high_order_I} and
      \eqref{eq:high_order_I_local} can be computed explicitly as
      $(I_{0}^0)^{n+1, l}$ and $T^{n+1,l}$ are already known with
      $\sigma^{n+1, 0}$ chosen as $\sigma^{n}$ (which is a function of
      $T^n$).
    \end{remark}

    The coefficients
    $ \tilde{\boldsymbol b} =(\tilde{b}_l) , {\boldsymbol b} = (b_l),
    \mathcal{A} = (a_{lj})$
    and $\tilde{\bf \mathcal{A}} = (\tilde{a}_{lj})$ can be presented
    with a double Butcher tableau as
\begin{equation}
  \label{eq:IMEX_coe}
  \begin{array}{c|c}
    \tilde{\bc}& \tilde{\bf \mathcal{A}} \\\hline 
               & \tilde{\boldsymbol b}^T
  \end{array}, \qquad
  \begin{array}{c|c}
    \bc & \bf \mathcal{A} \\\hline 
               & \boldsymbol b^T
  \end{array}.
\end{equation}
The second-order and third-order globally stiffly accurate IMEX
schemes used here are the ARS$(2,2,2)$ and ARS$(4, 4, 3)$ scheme,
where the exact Butcher tableaus are as below
\begin{equation}
  \label{eq:IMEX_2}
  \begin{array}{c|ccc}
    0& 0 &0 & 0  \\
    \gamma  & \gamma & 0 & 0 \\
    1 & \delta  & 1-\delta & 0  \\
    \hline
     & \delta & 1 - \delta & 0
  \end{array},
\qquad 
  \begin{array}{c|ccc}
    0& 0 &0 & 0  \\
    \gamma  & 0 & \gamma & 0  \\
    1 & 0  & 1-\gamma & \gamma  \\
    \hline
     & 0 & 1 - \gamma & \gamma
  \end{array}, \qquad \gamma = 1 - \frac{1}{\sqrt{2}}, \qquad  \delta
  = 1 - \frac{1}{2\gamma}, 
  \end{equation}
and 
\begin{equation}
  \label{eq:IMEX_3}
  \begin{array}{c|ccccc}
    0& 0 &0 & 0  & 0 & 0\\
    1/2 &  1/2 & 0 & 0 & 0 & 0 \\
    2/3 & 11/18 & 1/18 & 0 & 0 & 0 \\
    1/2 & 5/6 & -5/6 & 1/2 & 0 & 0 \\
    1 & 1/4 & 7/4 & 3/4 & -7/4 & 0  \\
    \hline
     & 1/4 & 7/4 & 3/4 & -7/4 & 0  
   \end{array},
\qquad 
\begin{array}{c|ccccc}
  0& 0 &0 & 0  & 0 & 0\\
  1/2 & 0 &  1/2 & 0 & 0 & 0 \\
  2/3 & 0 &  1/6 & 1/2& 0 & 0 \\
  1/2 & 0 &  -1/2 & 1/2 & 1/2 & 0  \\
  1 &  0 & 3/2 & -3/2 & 1/2 & 1/2  \\
  \hline
   &  0 & 3/2 & -3/2 & 1/2 & 1/2  \\
\end{array}.
\end{equation}

The IMEX RK methods are fully discussed in the literature, and we
refer to \cite{xiong2015, Jang2015, Peng2020} and the references
therein for more details.

\subsection{Fully discrete numerical scheme}
We have introduced the time-discretization in the last subsection. In
this subsection, the discretization in the spatial space will be
discussed. The finite volume method to discretize the spatial space is
presented here.

\subsubsection{Spatial Discretization}
The $P_N$ equations \eqref{eq:moment_equation} and
\eqref{eq:Pn_energy} are discretized by the finite volume method with
linear or third-order WENO reconstruction in space. Let
$x_i = i \Delta x, z_j = j \Delta z$ and $t^n = n \Delta t$ be the
uniform mesh in Cartesian coordinates. Let $(i, j)$ denote the cell
$\{(x, z): x_{i-1/2} < x < x_{i+1/2}, z_{j-1/2} < z <
z_{j+1/2}\}$. $(I_{i,j, l}^m)^n$ and $T_{i,j}^{n}$ are the averaged
expansion coefficients of the specific intensity and the temperature,
respectively. To get the numerical flux for $P_N$ system, we rewrite
the convection part of \eqref{eq:first_I0} and \eqref{eq:first_I}
together as
\begin{equation}
  \label{eq:con}
  \mathcal{C} =  \epsilon \left( \bar{\bA}_x^{\rm low} \pd{\bbI}{x}  +
    \bar{\bA}_z^{\rm low} \pd{\bbI}{z}\right)^{n+1}
  + \epsilon\left( \bar{\bA}_x^{\rm up} \pd{\bbI}{x}  
    + \bar{\bA}_z^{\rm up} \pd{\bbI}{z}\right)^{n}, 
\end{equation}
where $\bbA_x^{\rm low}$, $\bbA_z^{\rm low}$, $\bbA_x^{\rm up}$,
$\bbA_z^{\rm up}$ are made up by $F(I_{l-1}^m)$ and $G(I_{l+1}^m)$,
respectively with $\bbI = (I_0^0, I_1^{-1}, I_1^0, \cdots)$.  The
Lax-Friedrichs scheme is utilized here to obtain the numerical flux as
\begin{equation}
  \label{eq:flux}
  \begin{split}
    &\left( \epsilon \bbA_x^{s} \pd{\bbI}{x}\right)_{i,j}^{l} \approx
    \frac{1}{\Delta x} \Big(\mF^{s,l}_{i,j}\left(\bbI_{i,j}^{l},
      \bbI_{i+1, j}^{l}\right) - \mF^{s,l}_{i,j}\left(\bbI_{i-1,j}^{l},
      \bbI_{i, j}^{l}\right)\Big), \quad l = n, n+1,
    \quad s  = {\rm low, up},    \\
    & \left(\epsilon \bbA_z^{s} \pd{\bbI}{z}\right)_{i,j}^{l} \approx
    \frac{1}{\Delta z} \Big(\mG^{s,l}_{i,j}\left(\bbI_{i,j}^{l},
      \bbI_{i, j+1}^{l}\right) - \mG^{s,l}_{i,j}\left(\bbI_{i,j-1}^{l},
      \bbI_{i, j}^{l}\right)\Big),
    \quad l = n, n+1,   \quad s  = {\rm low, up}, \\
  \end{split}
\end{equation}
where the exact form of $\mF^{s,l}_{i,j}(U_1, U_2)$ and
$\mG^{s,l}_{i,j}(U_1, U_2)$ are
\begin{equation}
  \label{eq:flux_force}
  \begin{aligned}
    \mS_{i,j}^{s,l}(U_1, U_2) &= \frac{\epsilon}{2} { \bbA_w^{s}}(U_1
    + U_2) - {\rm coe}(s, l)\frac{ \alpha_{i,j} \epsilon
    }{2}(U_2-U_1), \qquad w = x, z, \qquad \mS = \mF, \mG, 
   \end{aligned}
 \end{equation}
 with
 \begin{equation}
   \label{eq:alpha_0}
  \alpha_{i,j} =   \alpha(\sigma_{i,j}, \epsilon) = \exp(-\sigma_{i,j} /\epsilon^2),
\end{equation}
and
\begin{equation}
  \label{eq:coe1}
  {\rm coe}(s, l) = \left\{
      \begin{array}{cc}
        1, &  s = {\rm up} ~{\rm or}~ (l = M~ {\rm and}~ s = {\rm low}), \\
        0, & {\rm otherwise.}
      \end{array}
    \right.
  \end{equation}
  The parameter ${\rm coe}(s, l)$ is utilized here to make sure that
  the diffusion term will only appear once in the numerical flux.

    \begin{remark}
      The coefficient $\alpha(\sigma, \epsilon)$ is an artificial
      parameter inspired by \cite{Control2020}, and is chosen to
      ensure the stability of the numerical scheme when $\epsilon$ is
      small.  The numerical scheme will reduce to the central
      difference scheme with $\epsilon$ going to zero, and remains to
      be the Lax-Friedrichs numerical scheme when $\epsilon$ is
      large. We have proved the stability of the first-order numerical
      scheme with this $\alpha$ in the following sections by Fourier
      analysis and energy stability analysis. More discussions and
      analyses about $\alpha$ will be also done in future work.
  \end{remark}

\subsubsection{Time step length}
As will be proved in Section \ref{sec:ap_analysis}, when $\epsilon$
goes to zero, the numerical scheme \eqref{eq:first_order} with
\eqref{eq:flux} will converge to an explicit scheme of the nonlinear
diffusion equation \eqref{eq:limit}. Therefore, the time step
length is set as
\begin{equation}
  \label{eq:time}
  \Delta t = \max\{C \epsilon \Delta x /c, C \sigma_{\rm min}\Delta
  x^2 / c  \}, 
\end{equation}
where $\sigma_{\rm min}$ is the minimum value of $\sigma(\bx)$ all
over the computation domain, and $C$ is the CFL number. Here, it
always requires that $C < 1$.

Following the method in \cite{Peng2020}, we will discuss the stability
of the numerical scheme \eqref{eq:first_order} and \eqref{eq:flux}
with the time step length \eqref{eq:time} for the linear system
\eqref{eq:linear_RTE}. Numerical experiments indicate that this choice
may also work when the method is applied to more general models, such
as the gray approximation to the radiative transfer equations
\eqref{eq:RTE}.  Let $P_1$ system as an example. The first-order
system for the $P_1$ system is reduced into
  \begin{equation}
    \label{eq:P1}
    \begin{aligned}
      \epsilon^2 \frac{I_{0,j}^{n+1} - I_{0,j}^n}{\Delta t} & +
      \epsilon \frac{I_{1, j+1}^{n} - I_{1, j-1}^{n}}{2 \Delta x} -
      \frac{\alpha \epsilon }{2} \frac{I_{0, j+1}^n - 2 I_{0,j}^n +
        I_{0, j-1}^n}{\Delta x}  = 0,  \\
      \epsilon^2 \frac{I_{1,j}^{n+1} - I_{1,j}^n}{\Delta t} & +
      \frac{\epsilon}{3} \frac{I_{0, j+1}^{n+1} - I_{0, j-1}^{n+1}}{2
        \Delta x} - \frac{\alpha \epsilon }{2} \frac{I_{1, j+1}^n - 2
        I_{1,j}^n +
        I_{1, j-1}^n}{\Delta x}  = - I_{1,j}^{n+1}.
  \end{aligned}
\end{equation}
We follow the definition of stability as in \cite{Peng2020} here. To
carry out the Fourier analysis, assuming the mesh is uniform and the
periodic boundary condition is imposed, let the numerical solutions be
\begin{equation}
  \label{eq:Fourier_scheme}
  (I_0)_j^n = \hat{I}_0^n \exp(\imag k j h),  \qquad (I_1)_j^n =
  \hat{I}_1^n \exp(\imag k j h),\qquad j = 0, \cdots, N-1, \qquad k \in \bbZ,
\end{equation}
where $j$ is the index in the $x-$axis and $k$ is the index for the
Fourier mode. Then, \eqref{eq:P1} will be reduced to
\begin{equation}
  \label{eq:fourier_mat}
  \left(
    \begin{array}[c]{c}
      \hat{I}_0^{n+1} \\ 
      \hat{I}_1^{n+1}
    \end{array}
  \right)  = {\bf C}(\epsilon, \alpha, \Delta t, \Delta x, \xi) \left(
    \begin{array}[c]{c}
      \hat{I}_0^{n} \\ 
      \hat{I}_1^{n}
    \end{array}
  \right)
\end{equation}
with
\begin{equation}
  \label{eq:fourier_g}
  {\bf C}(\epsilon, \alpha, \Delta t, \Delta x, \xi) = \left( \begin{array}[c]{cc}
                                                                \frac{\epsilon^2}{\Delta
                                                                t} &  0\\ 
                                                                \frac{\imag
                                                                \epsilon
                                                                \sin(\xi)}{
                                                                3
                                                                \Delta
                                                                x}
                                                                   &
                                                                     \frac{\epsilon^2}{\Delta t} + 1
    \end{array}\right)^{-1} \left( \begin{array}[c]{cc}
      \frac{\epsilon^2}{\Delta t} + \frac{\alpha \epsilon (\cos(\xi)
                                     -1)}{\Delta x} &
                                                               -\frac{
                                                     \imag \epsilon
                                                      \sin(\xi)}{\Delta
                                                      x   } \\
                                     0 &  \frac{\epsilon^2 }{\Delta t}
                                         + \frac{\alpha \epsilon
                                         (\cos(\xi) -1)}{\Delta x}
    \end{array}\right), \qquad \xi = k \Delta h \in [0, 2\pi].
\end{equation}
As is stated in \cite{Peng2020}, the numerical scheme \eqref{eq:P1} is stable, if it satisfies $\forall \xi \in [0, 2\pi],$
\begin{enumerate}
\label{eq:linear_stability}
\item $\max\{|\lambda_1(\xi)|, |\lambda_2(\xi)|\} < 1$,
\item $\max\{|\lambda_1(\xi)|, |\lambda_2(\xi)|\} =1$ and $\bf C$ is real diagonalizable, 
\end{enumerate}
where $\lambda_i, i = 1, 2$ are the eigenvalues of $\bf C$.  As is
stated in \cite{Peng2020}, the stability is a necessary condition for
the standard $L^2$ energy to be non-increasing.

\begin{proposition}
  \label{thm:fourier}
  The numerical scheme \eqref{eq:P1} for the linear system
  \eqref{eq:linear_RTE} satisfies the conditions
  \eqref{eq:linear_stability}, and is then stable.
\end{proposition}

This proposition is proved in Appendix \ref{app:fourier}, and can be
extended to the general $P_N$ system.

 \subsubsection{Algorithm}
 \label{sec:algorithm}
 Based on all the discussions above,  the algorithm will be summarized as below.
 \begin{enumerate}
 \item Given $\bbI_{i,j}^n$ and $T_{i, j}^n$ at time step $n$;
 \item Update the specific intensity and temperature according to the
   IMEX scheme \eqref{eq:high_order}, \eqref{eq:high_order_local} and
   \eqref{eq:flux}, which includes the two steps below at the internal
   stage $k$ of a RK step
  \begin{enumerate}
  \item Obtain $T_{i,j}^{k}$ and $(I_{0}^0)_{i,j}^{k}$ by solving the
    equations \eqref{eq:high_order_I0_local} and
    \eqref{eq:high_order_T_local};
  \item Calculate $(I_{l}^m)_{i,j}^{k}$ by \eqref{eq:high_order_I_local};
  \end{enumerate}
\item Go to 1 for the next step.
  \end{enumerate}

\section{Formal asymptotic property and stability analysis}
\label{sec:proof}
In this section, we will study the asymptotic property and the
stability for the proposed AP IMEX scheme.

\subsection{Formal asymptotic analysis}
\label{sec:ap_analysis}

The asymptotic preserving property is quite important for multi-scale
problems. In the realistic thermal radiative transfer problems, it is
not practical to resolve the mean-free path, which requires
prohibitively small grid cells. Therefore, the AP property is
required. It is expected that when holding the mesh size and time step
fixed, the AP scheme should automatically recover the discrete
diffusion solution when the mean free path goes to zero
\cite{sun2015asymptotic1, Klar1998, Larsen1987, Larsen1989,
  Mie2013}. For the radiative transfer problem \eqref{eq:RTE}, this is
to say that the numerical method could give a valid discretization of
the nonlinear diffusion equation \eqref{eq:limit} \cite{semi2008Ryan}.

We will examine the AP property of this numerical method in the asymptotic limit
away from boundary and initial layers, and the first-order numerical
scheme is studied here. The theorem below shows the AP property of
this method.

\begin{theorem}
  As the parameter $\epsilon$ goes to zero, the numerical scheme
  proposed in \ref{sec:algorithm} approaches an explicit five-point
  scheme for the nonlinear diffusion equation \eqref{eq:limit}.
\end{theorem}
\begin{proof}
  From \eqref{eq:order_l} and \eqref{eq:first_I}, it holds that
  \begin{equation}
    \label{eq:order_I1_1_1}
    -\sigma^{n+1} (I_1^m)^{n+1} \approx  \epsilon \Bigg(\nabla_{\bx} F(I_0^m)\Bigg)^{n+1} = 
    \epsilon \left(\frac{\partial}{2\partial x}\Big(
      -C_{0}^{m-1}I_{0}^{m-1} + \mE_{0}^{m+1} I_{0}^{m+1} \Big) +
      \pd{}{z} A_{0}^m I_{0}^m
    \right)^{n+1}, \qquad
    |m| \leqslant 1. 
  \end{equation}
  With the numerical flux \eqref{eq:flux}, we can derive the final
  approximation to $(I_1^m)_{i,j}^{n+1}$ at grid $(i,j)$. Precisely,
  the exact expression for $(I_1^m)_{i,j}^{n}, m = -1, 0, 1$ is
\begin{equation}
  \label{eq:AP_I1}
  \begin{aligned}
    (I_1^0)_{i,j}^{n} & = \frac{-\epsilon}{\sigma_{i,j}^{n}
    }\left(\sqrt{\frac{1}{3}} \frac{(I_{0}^0)_{i,j+1}^n -
        (I_0^0)_{i,j-1}^n}{\Delta z} + \frac{ \alpha_{i,j}^n
        }{2c }\frac{ \big((I_1^0)_{i,j+1}^{n}
        -    2(I_1^0)_{i,j}^n+(I_{1}^{0})_{i,j-1}^n)}{\Delta z} \right), \\
    (I_1^{-1})_{i,j}^{n} & = \frac{-\epsilon}{\sigma_{i,j}^{n}
    }\left(\sqrt{\frac{2}{3}} \frac{(I_0^0)_{i+1,j}^{n} -
        (I_0^0)_{i-1,j}^{n}}{2\Delta x} + \frac{\alpha_{i,j}^n}{2
        c}\frac{(I_{1}^{-1})_{i+1,j} ^{n} - 2
        (I_{1}^{-1})_{i,j}^{n} + (I_1^{-1})_{i-1, j}^{n} }{\Delta x} \right), \\
    (I_1^1)_{i,j}^{n} & = \frac{\epsilon}{\sigma_{i,j}^{n}
    }\left(\sqrt{\frac{2}{3}} \frac{(I_0^0)_{i+1,j}^{n} -
        (I_0^0)_{i-1,j}^{n}}{2\Delta x} + \frac{\alpha_{i,j}^n}{2
        c }\frac{(I_{1}^1)_{i,j}^{n} - 2 (I_1^1)_{i,j}^{n} +
       (I_1^1)_{i-1, j}^{n} }{\Delta x} \right).
\end{aligned}
\end{equation}
With the numerical flux \eqref{eq:flux}, the fully discrete form of
\eqref{eq:first_T} is reduced into
\begin{equation}
  \label{eq:reduced_T}
  C_{v} \frac{T_{i,j}^{n+1} - T_{i,j}^n}{\Delta t} +
  \frac{1}{c}\frac{(I_0^0)_{i,j}^{n+1} - (I_0^0)_{i,j}^n}{\Delta t}  +
  \frac{1}{\epsilon}   \left(\frac{\mF_{i+1/2,j}^n -
      \mF_{i-1/2, j}^n}{\Delta x} + \frac{\mG_{i,j+1/2}^n -
      \mG_{i,j-1/2}^{n}}{\Delta z}\right)  = 0,
\end{equation}
where
\begin{equation}
  \label{eq:AP_flux}
  \begin{aligned}
    &\mF_{i+1/2,j, g}^n = \frac{1}{4}\sqrt{\frac{2}{3}}
    \Big(\left((I_1^{-1})_{i,j}^{n} - (I_1^1)_{i,j}^{n}\right) +
    \left((I_1^{-1})_{i+1,j}^{n} - (I_1^1)_{i+1,j}^{n}\right) \Big) +
    \frac{\alpha_{i,j}^n }{4c}\Big((I_0^0)_{i+1,j}^{n} -
    (I_0^0)_{i,j}^n\Big), \\
    & \mG_{i,j+1/2, g}^n = \frac{1}{2\sqrt{3}} \Big((I_1^0)_{i,j}^{n}
    + (I_1^0)_{i,j+1}^{n} \Big)+ \frac{\alpha_{i,j}^n}{4c}\Big( (I_0^0)_{i,j+1}^{n} - (I_0^0)_{i,j}^n\Big).
  \end{aligned}
\end{equation}
From \eqref{eq:alpha_0}, we can obtain that 
\begin{equation}
  \label{eq:limit_alpha}
  \lim_{\epsilon \rightarrow 0} \alpha_{i,j}^n =  \lim_{\epsilon
    \rightarrow 0} \exp(-\sigma_{i,j}^n/\epsilon^2) = 0.  
\end{equation}
Substituting \eqref{eq:AP_I1} into \eqref{eq:AP_flux} and omitting the
higher-order term of $\epsilon$, it holds that
\begin{equation}
  \label{eq:AP_Flux_T}
  \begin{aligned}
    &\frac{1}{\epsilon} \left(\frac{\mF_{i+1/2,j, g}^n - \mF_{i-1/2,
          j, g}^n}{\Delta x} + \frac{\mG_{i,j+1/2,g}^n -
        \mG_{i,j-1/2,g}^{n}}{\Delta z}\right)  \\
    & = -\frac{1}{3} \left(\frac{ \frac{(I_0^0)_{i+2,j}^{n} - (I_0^0)_{i,j}^{n}
      }{2\Delta x\sigma_{i+1,j}^n} - \frac{ (I_0^0)_{i,j}^{n}
        -(I_0^0)_{i-2,j}^{n} }{2\Delta x\sigma_{i-1,j}^n}}{2\Delta
      x} + \frac{ \frac{(I_0^0)_{i,j+2}^{n} -(I_0^0)_{i,j}^{n} }{2\Delta
        z\sigma_{i, j+1}^n} - \frac{(I_0^0)_{i,j}^{n}
        -(I_0^0)_{i,j-2}^{n} }{2\Delta z\sigma_{i,j-1}^n}}{2\Delta
      z}\right).
 \end{aligned}  
\end{equation}
Together with \eqref{eq:AP_flux}, and \eqref{eq:AP_Flux_T}, we can
find that \eqref{eq:reduced_T} becomes a five-point scheme for the
nonlinear diffusion equation \eqref{eq:limit}, and this shows that the
current scheme for RTE \eqref{eq:RTE} is an AP scheme.
\end{proof}

For the 1D spatial problem, when $\epsilon$ goes to zero,
\eqref{eq:reduced_T} is reduced into
\begin{equation}
  \label{eq:1D_limit}
   C_{v} \frac{T_{i}^{n+1} - T_{i}^n}{\Delta t} +a
  \frac{(T_{i}^4)^{n+1} -
    (T_{i}^4)^n}{\Delta t}   =  \frac{ac}{3\sigma} \frac{ (T_{i+2}^4)^{n}
    -2(T_{i}^4)^{n}  + (T_{i-2}^4)^{n}}{4 \Delta x^2 }.
\end{equation}
The linear version of \eqref{eq:1D_limit} is
\begin{equation}
  \label{eq:1D_limit_linear}
  \frac{\phi_i^{n+1} -
    \phi_i^n}{\Delta t}   =  \frac{c}{3\sigma} \frac{ \phi_{i+2}^{n}
    -2\phi_{i}^{n}  + \phi_{i-2}^{n}}{4 \Delta x^2 }, \qquad \phi_i
  = T_i^4.
\end{equation}
Fourier analysis shows the stability condition for \eqref{eq:1D_limit_linear} is 
\begin{equation}
  \label{eq:stability}
  \frac{c \Delta t}{3 \sigma \Delta x^2}
  \leqslant  \frac{1}{\max| \cos(2 \Delta x) - 1|} = 1.
\end{equation} 

\subsection{Energy stability}
Energy stability is another important property for a new numerical
scheme. In this section, the general nonlinear stability of the
numerical scheme for the complete system will be discussed.

\subsubsection{Gray approximation of the radiative transfer equations}
Just to show the property of the numerical scheme, $P_N$ equations for
the gray approximation of the radiative transfer equations in a
one-dimensional planar geometry medium are studied. The exact form of
the gray approximation of the radiative transfer equations and the
corresponding $P_N$ equations are presented in Appendix
\ref{app:1D}. Without loss of generality, the opacity $\sigma(x, T)$
and heat capacity $C_{v}$ are all set as constant. The periodic
boundary condition is adopted in the spatial space. In this case, the
first-order scheme \eqref{eq:first_order} is reduced
into \begin{subequations}
  \label{eq:1D_first_order}
  \begin{align}
  \label{eq:1D_first_I0}
    & \frac{\epsilon^2}{c} \frac{\bbI_{i}^{n+1} - \bbI_{i}^n }{\Delta t}
      + \Gamma^{\rm low}(\bbI_{i}^{n+1})  + 
      \Gamma^{\rm up}(\bbI_{i}^n)  = -\sigma 
      \bbI_i^{n+1} +   a c  \sigma\left(T^4\right)_i^{n+1} e_1, \\
    \label{eq:1D_first_T}
    & \epsilon^2 C_{v} \frac{T_i^{n+1} - T_i^n}{\Delta t}
       + \frac{\epsilon^2}{c} \frac{I_{0,i}^{n+1} - I_{0, i}^n }{\Delta t}
      +   \Gamma^{\rm up}(\bbI_{i}^n)e_1
      =  0.
  \end{align}
\end{subequations}
where the numerical flux is defined as
\begin{equation}
  \label{eq:Gamma}
  \Gamma^s(U_i) = \frac{\mF^s(U_{i+1}, U_{i}) - \mF^s(U_i, U_{i-1}) }{\Delta
    x}, \qquad s = \rm up, ~low, 
\end{equation}
with $\mF^{s}(U_1, U_2)$ defined in \eqref{eq:flux_force}, $e_1$
defined in Appendix \ref{app:1D}, and the matrix $\bbA_x^s$ changed
into $\bB^s$ defined in Appendix \ref{app:1D}.

We will begin from the gray approximation of the radiative transfer
equations to establish the numerical stability analysis for
\eqref{eq:1D_first_order}. The proposition below shows the energy
inequality for this system as
\begin{proposition}
  \label{pro:1D_RTE_inequality}
  For the RTE \eqref{eq:1D_RTE} with constant opacity $\sigma$ and
  periodic boundary in the spatial space, the energy inequality
  holds
  \begin{equation}
    \label{eq:energy_stability}
    \frac{\epsilon^2}{2c} \pd{}{t} \int_{x\in L}\int_{-1}^{1} I^2
    \dd \mu \dd x + \frac{\epsilon^2 C_{v}}{5} \pd{}{t}\int_{x\in
      L} \frac{1}{2} ac T^5  \dd x \leqslant 0. 
\end{equation}
\end{proposition}
\begin{proof}
  Multiplying \eqref{eq:1D_RTE_I} with $I$ and taking integration over
  $\mu$ and $x$, with the periodic boundary condition, we can derive that
  \begin{equation}
    \label{eq:energy_I}
    \frac{\epsilon^2}{2c} \pd{}{t} \int_{x\in L}\int_{-1}^{1} I^2
    \dd \mu \dd x  =
    \sigma\int_{x\in L} \left( \frac{1}{2}a c   T^4 \int_{-1}^1   I \dd
      \mu- \int_{-1}^1 I^2 \dd \mu\right) \dd x.
  \end{equation}
  Multiplying \eqref{eq:1D_RTE_T} with $\frac{ac}{2} T^4$ and
  integrating over $x$, it holds that 
\begin{equation}
  \label{eq:energy_T}
  \frac{\epsilon^2 C_{v}}{5} \pd{}{t}\int_{x\in L} \frac{1}{2} ac T^5  \dd x = \sigma
  \int_{x\in L} \left( \frac{1}{2} ac T^4\int_{-1}^1 I \dd \mu  - \frac{1}{2} (ac
    T^4)^2  \right) \dd x. 
\end{equation}
Together with \eqref{eq:energy_I} and \eqref{eq:energy_T}, it holds
with Cauchy-Schwarz inequality that
\begin{equation}
  \label{eq:energy_inequ}
  \begin{aligned}
    & \frac{\epsilon^2}{2c} \pd{}{t} \int_{x\in L}\int_{-1}^{1} I^2
    \dd \mu \dd x + \frac{\epsilon^2 C_{v}}{5} \pd{}{t}\int_{x\in
      L} \frac{1}{2} ac T^5 \dd x\\
    & \qquad = \sigma \int_{x\in L} \left(ac T^4\int_{-1}^1 I \dd \mu
      - \frac{1}{2} (ac
      T^4)^2- \int_{-1}^1 I^2 \dd \mu  \right) \dd x \\
    &\qquad \leqslant \frac{-\sigma}{2} \int_{x\in L} \left(ac T^4 -
      \int_{-1}^1 I \dd \mu \right)^2 \dd x \leqslant 0.
\end{aligned}
\end{equation}
Then we finish the proof of Proposition \ref{pro:1D_RTE_inequality}.
\end{proof}

Based on the energy inequality for the continuous equations, the
stability result for the first-order scheme \eqref{eq:1D_first_order}
is listed in the theorem below.
\begin{theorem}
  \label{thm:discrete_energy}
Following \cite{jang2014analysis}, define the discrete energy as
  \begin{equation}
    \label{eq:dis_energy}
    E(t^{n+1}) =   \sum_{i}\left[\frac{\epsilon^2}{2c \Delta t}
      \left(\left(I_{0,i}^{n+1}\right)^2 + \sum_{l=1}^M(2l
        +1)\left(I_{l,i}^{n}\right)^2 \right) + \epsilon^2 C_{v}
     \frac{ (T_i^5)^{n+1}}{5 \Delta t}\right]. 
  \end{equation}
  Then, for periodic boundary conditions, the following stability
  result holds for the first-order AP scheme defined in (4.7),
  \begin{equation}
    \label{eq:discrete_energy}
    E(t^{n+1}) - E(t^n) \leqslant 0
  \end{equation}
 with the time step length \eqref{eq:time}.

\end{theorem}

Due to the tedious process, the proof is put in Appendix
\ref{app:energy_proof}.


\section{Numerical results}
\label{sec:num}
In this section, several numerical simulations for the radiative
transfer equations in spatially 1D and 2D cases are studied to
validate the efficiency of this numerical method. We have implemented
the first-order \eqref{eq:first_order} and third-order IMEX RK scheme
\eqref{eq:high_order} to approximate RTE. In all 1D numerical tests,
the CFL number is set as $C = 0.4$. For 2D test cases, the CFL number
is set as $C = 0.1$. Periodic and inflow boundary conditions are
implemented, the details of which are presented in Appendix
\ref{app:boundary_condition}.

\subsection{The AP property}
This example is designed to test the AP property and the order of
accuracy of this numerical method for the first-order scheme
\eqref{eq:first_order} and the higher-order scheme
\eqref{eq:high_order}.  The test starts with an equilibrium initial
data
\begin{equation}
  \label{eq:initial_ex1}
  T =  (3 + \sin(\pi  x) ) / 4, \qquad I = ac T^4, \qquad x \in L. 
\end{equation}
The computation region is set as $L = [0, 2]$ with the periodic
boundary condition imposed on both ends. The parameters are set as
$a = c = 1.0$, $C_{v} = 0.1$ and $\sigma = 10$. Similar tests can be
found in the literature for the Boltzmann equation \cite{JinYan2013}.

\begin{figure}[!htb]
  \centering
  \includegraphics[width=0.4\textwidth]{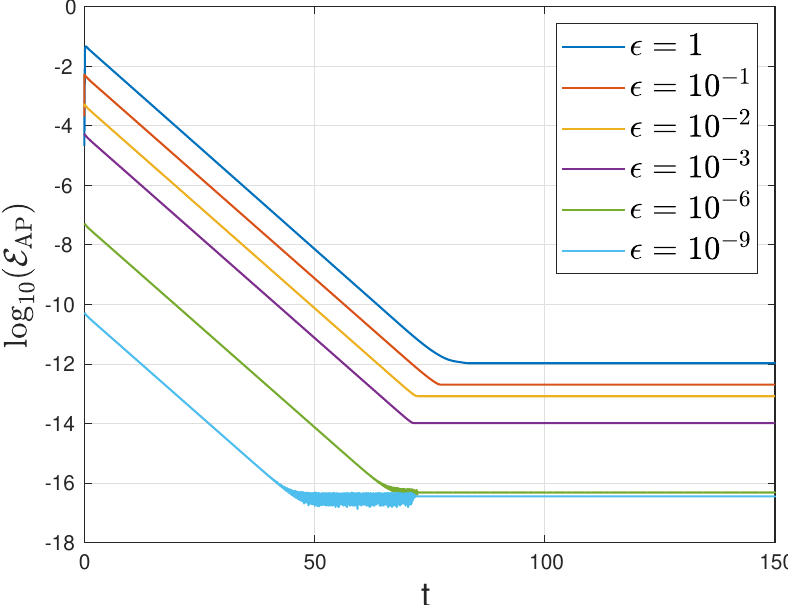}
  \caption{Time evolution of $\mE_{\rm AP}$ with different
    $\epsilon$. The $x$-axis is time $t$, and the $y$-axis is
      $\log_{10}(\mE_{\rm AP})$.}
  \label{fig:AP}
\end{figure}

In this test, the mesh size is $N = 100$ and the expansion order of
the $P_N$ method is $M = 7$. Since we are going to test the behavior
of the numerical scheme when $\epsilon$ goes to zero, the time step is
set as $\Delta t = {\rm C} \Delta x^2 / c$. Figure \ref{fig:AP} shows
the time evolution of $\mE_{\rm AP}$ as
\begin{equation}
  \label{eq:error_AP}
  \mE_{\rm AP} = \sqrt{\Delta x \sum_{i=1}^N \Big((I_{0,i} - a c T_{i}^4)^2 + \sum_{j=1}^M 
    I_{j,i}^2\Big)}, 
\end{equation}
for the numerical scheme \eqref{eq:first_order} with different
$\epsilon$.  We can see that for any $\epsilon$, $\mE_{\rm AP}$ is
decreasing with time and then reaches a final steady state.  With the
decreasing of $\epsilon$, the final value of $\mE_{\rm AP}$ becomes
smaller, which shows the AP property of the numerical scheme.

Next, we test the order of the numerical scheme
\eqref{eq:first_order}. The initial data \eqref{eq:initial_ex1} with
the same parameters are applied. We compute the solutions with grid
size $N = 50, 100, 200, 400$ and $800$, respectively for
$\epsilon = 1, 0.1$ and $0.01$. The final time is $t = 0.5$, and the
numerical solution with $N = 1600$ is chosen as the reference
solution. The $l_2$ error between the numerical solution and the
reference solution with different $\epsilon$ is calculated. Figure
\ref{fig:order_error} shows the convergence order of the numerical
method for different $\epsilon$. It illustrates that for different
$\epsilon$, the scheme is the uniformly first-order, which also
validates the AP property of the numerical scheme.

\begin{figure}[!htb]
  \centering
  \subfloat[$I$]{
    \includegraphics[width=0.4\textwidth]{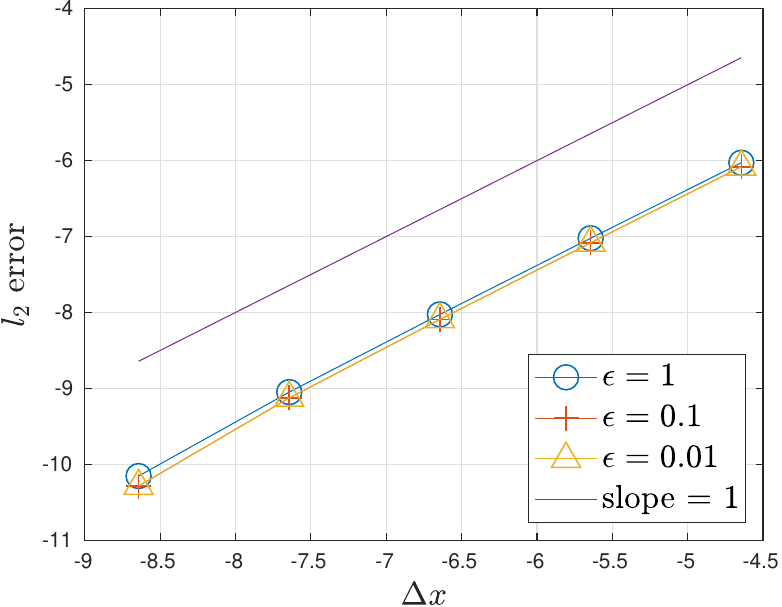}}
  \hfill
  \subfloat[$T$]{
    \includegraphics[width=0.4\textwidth]{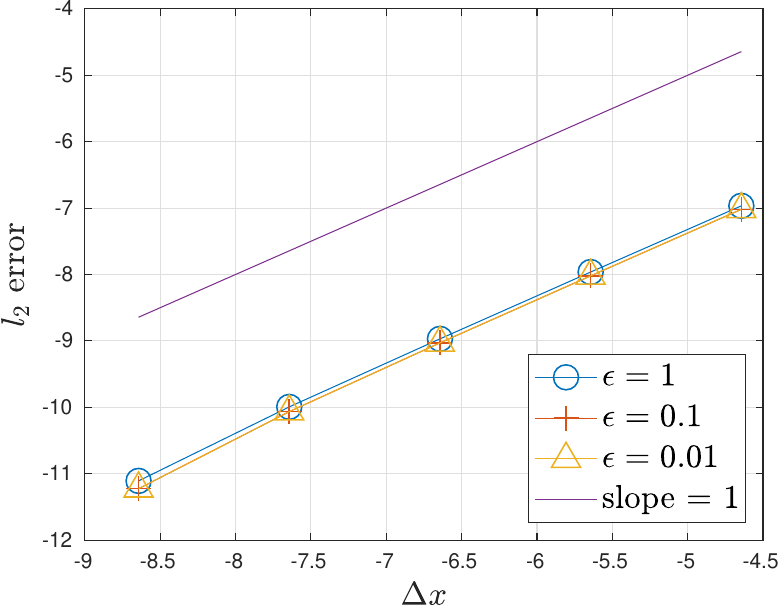}}
  \caption{$l_2$ error between the numerical solution with grid size
    $N = 50, 100, 200, 400$ and $800$ and the reference solution
    $N = 1600$.  (a) The $l_2$ error of the specific intensity
    $I$. (b) The $l_2$ error of the material temperature $T$.}
  \label{fig:order_error}
\end{figure}

To further verify the AP property of this numerical scheme, we redo
the test with higher-order scheme. First, the linear reconstruction
with IMEX3 scheme is utilized. The grid size is set as
$N = 100, 200, 400, 800$, respectively for
$\epsilon = 1, 0.1, 0.01, 10^{-6}$. The numerical solution with
$N = 1600$ is chosen as the reference solution. The final time is
$ t = 0.5$, and the $l_2$ error between the numerical solution and the
reference solution with different $\epsilon$ is calculated. Figure
\ref{fig:order_error_linear} shows the convergence order of the
numerical method for different $\epsilon$. It illustrates that for
different $\epsilon$, the scheme is the uniformly second-order.

\begin{figure}[!htb]
  \centering
  \subfloat[$I$]{
    \includegraphics[width=0.4\textwidth]{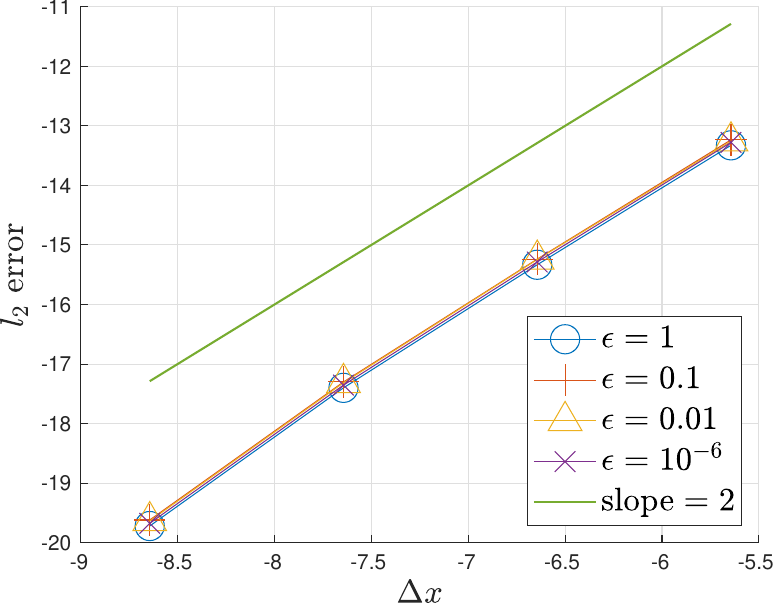}}
  \hfill
  \subfloat[$T$]{
    \includegraphics[width=0.4\textwidth]{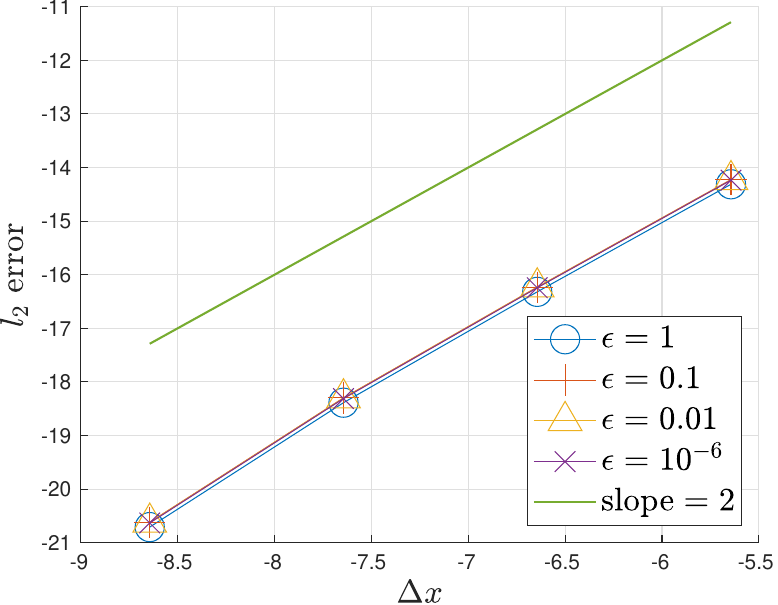}}
  \caption{$l_2$ error with linear reconstruction between the numerical solution with grid size
    $N = 100, 200, 400$ and $800$ and the reference solution is 
    $N = 1600$.  (a) The $l_2$ error of the specific intensity
    $I$. (b) The $l_2$ error of the material temperature $T$.}
  \label{fig:order_error_linear}
\end{figure}

With the same settings as above, we will test the convergence order of
the numerical method for different $\epsilon$ by IMEX3 scheme with the
third-order WENO reconstruction. The numerical results are shown in
Figure \ref{fig:order_error_weno}. We could see that the convergence
order of different $\epsilon$ is the same. However, it is only
second-order. We have studied the reason carefully, and a simple proof
is proposed in Appendix \ref{app:lossorder_proof}.
  
\begin{figure}[!htb]
	\centering
	\subfloat[$I$]{
		\includegraphics[width=0.4\textwidth]{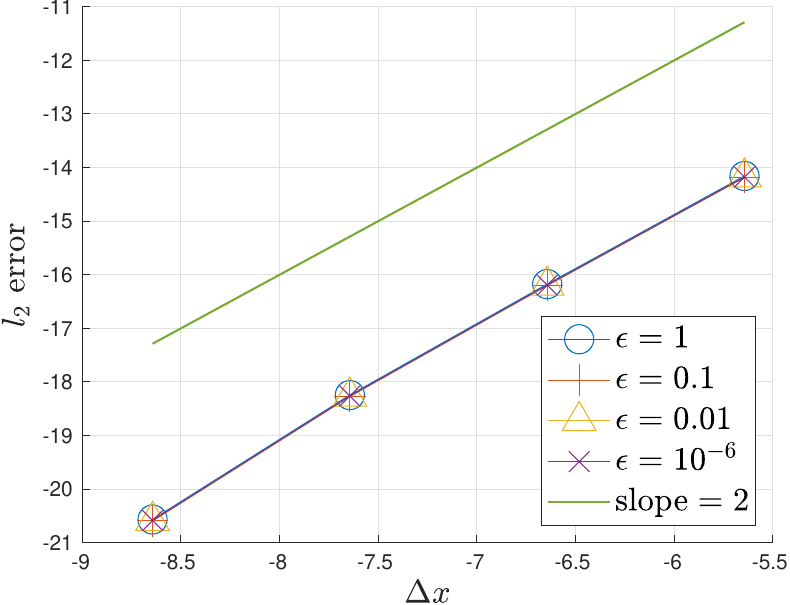}}
	\hfill
	\subfloat[$T$]{
		\includegraphics[width=0.4\textwidth]{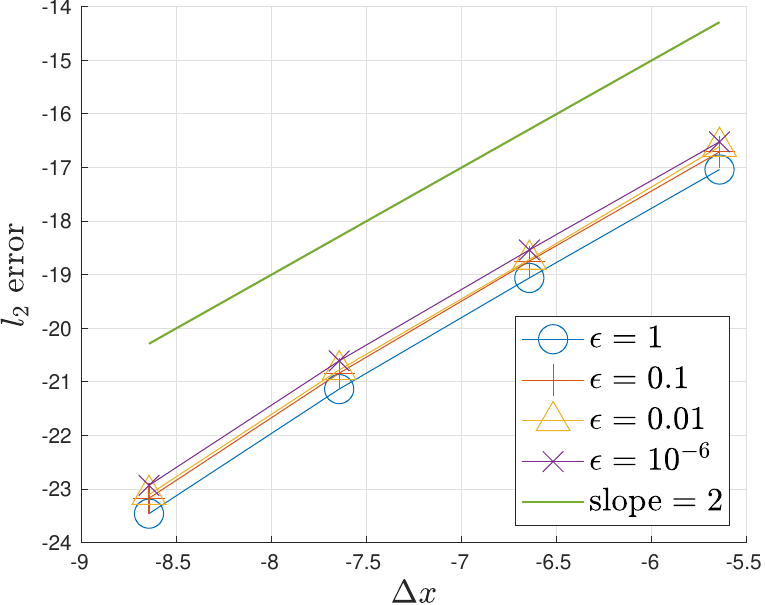}}
	\caption{$l_2$ error with third-order WENO reconstruction between the numerical solution with grid size
		$N = 100, 200, 400$ and $800$ and the reference solution
		$N = 1600$.  (a) The $l_2$ error of the specific intensity
		$I$. (b) The $l_2$ error of the material temperature $T$.}
	\label{fig:order_error_weno}
\end{figure}

To record the evolution of the temperature $T$ with $\epsilon$, the
evolution of $T$ for different $\epsilon$ is plotted in Figure
\ref{fig:T_temp}, where two positions $x = 0.505$ and $1.005$ are
recorded. Here, the grid size is $N = 200$, and
$\epsilon = 1, 0.5, 10^{-2}, 10^{-6}$. The evolution of the
temperature $T$ of the nonlinear diffusion equation \eqref{eq:limit}
is also plotted. From it, we can see that the temperature $T$ is
converging to the solution of the nonlinear diffusion equation as
$\epsilon$ approaches zero. The behavior of $T$ also validates the
stability of the numerical scheme when $\epsilon$ goes to zero.
 
\begin{figure}[!htb]
	\centering
	\subfloat[$x = 0.505$]{
		\includegraphics[width=0.4\textwidth]{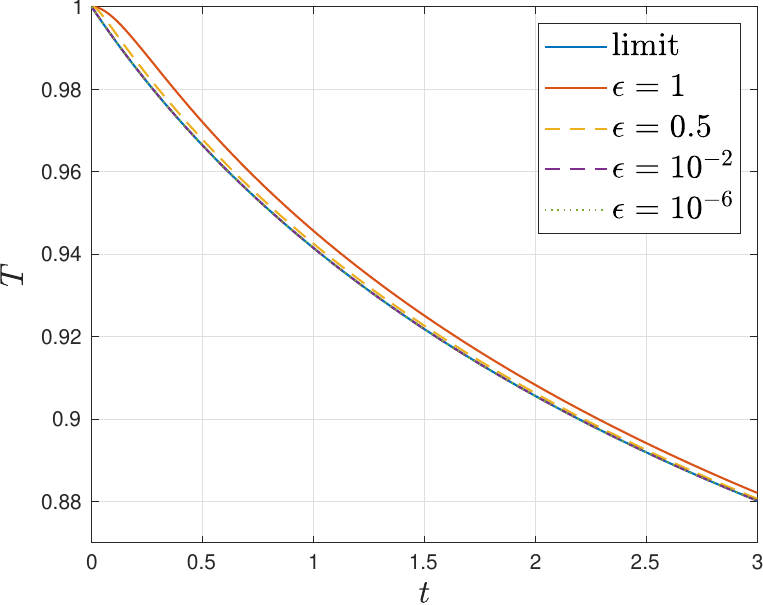}}
	\hfill
	\subfloat[$x = 1.005$]{
		\includegraphics[width=0.4\textwidth]{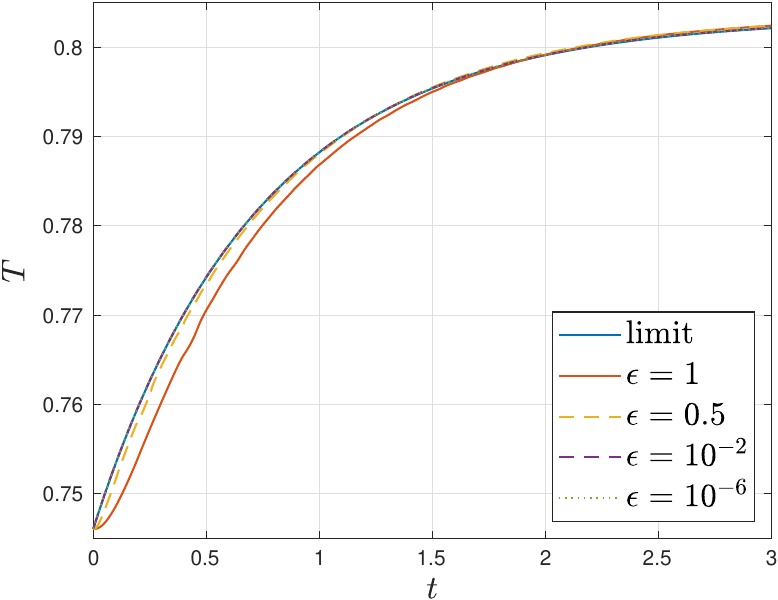}}
      \caption{The evolution of temperature $T$ with time increasing
        for different $\epsilon$. (a) The value of $T$ at $x =
        0.505$. (b) The value of $T$ at $x = 1.005$.}
	\label{fig:T_temp}
\end{figure}

\subsection{Marshak wave problems}
In the following examples, the classical Marshak wave problems are
tested. The Marshak problem is one of the benchmark problems and is
also studied in the literature such as \cite{sun2015asymptotic1,
  Larsen2013, semi2008Ryan}.  In the computations, the parameters are
chosen the same as that in \cite{sun2015asymptotic1} with
$a = 0.01372 \rm{GJ/cm^3-ke V^4}$ and $c = 29.98 \rm{cm / ns}$.  In
this section, two absorption/emission coefficients are tested. For
both cases, the inflow boundary condition is imposed on both the left
and right sides, where the Marshak type boundary condition
\cite{semi2008Ryan} is utilized. The details of the Marshak type
boundary are proposed in Appendix \ref{app:boundary_condition}.

\paragraph{Marshak Wave-2B}

\begin{figure}[!htb]
  \centering \subfloat[material temperature]{
    \includegraphics[width=0.4\textwidth]{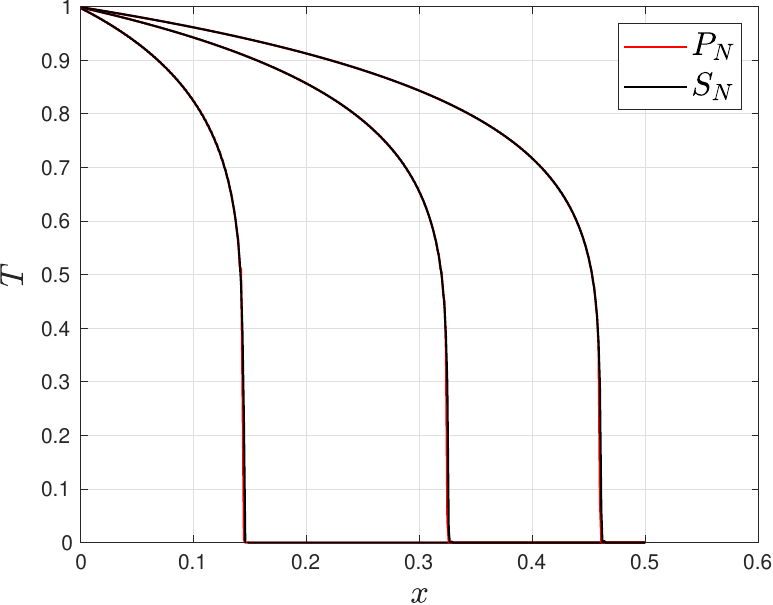}
    \label{fig:2B_1}}
  \hfill \subfloat[material temperature]{
    \includegraphics[width=0.4\textwidth]{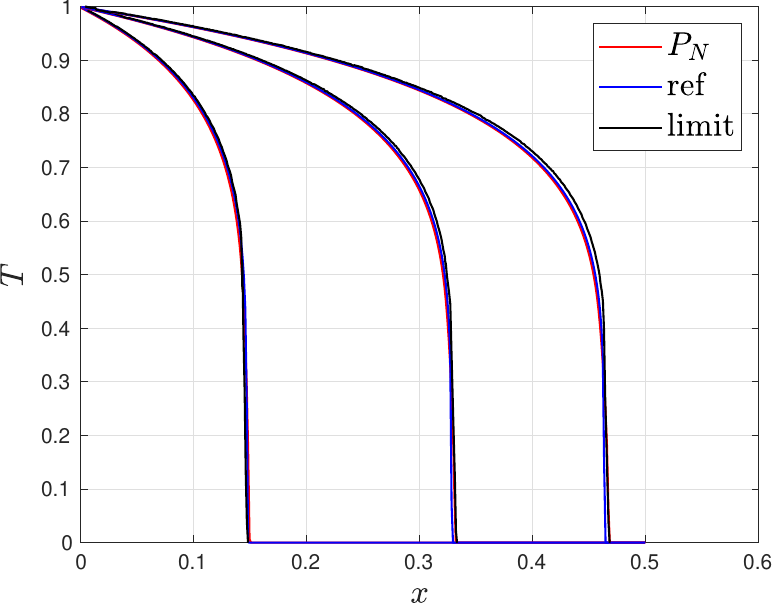}
    \label{fig:2B_2}}
  \caption{The material temperature $T$ of Marshak Wave-2B problem at time
    $t = 10, 50$ and $100$. The left picture is the $P_N$ solution and
    the reference solution get by $S_N$ method. The right picture is
    the $P_N$ solution and the reference solution in
    \cite{semi2008Ryan} and the black line is that to  the diffusion
    limit. }
  \label{fig:2B}
\end{figure}

In this example, we take the absorption/emission coefficient to be
$\sigma = 100/ {\rm T^3 cm^2/g}$, the density to be $3.0 \rm {g/cm^3}$
and the specific heat to be $0.1 \rm {GJ/g/keV}$.  The initial
material temperature $T$ is set to be $10^{-6} {\rm keV}$. A constant
isotropic incident specific intensity with a Planckian distribution at
$1$ keV is kept on the left boundary. The computation domian is 
$[0, \infty)$ but taken to be $L = [0, 0.2]$ in the simulations. In
this case, $\sigma$ is large enough that the solution to RTE is almost
the same as that of the diffusion limit \eqref{eq:limit}.

In the test, the expansion order of $P_N$ is set as $M = 11$ with the
grid size $N = 400$. The third-order IMEX RK scheme
\eqref{eq:high_order} is applied here, where the time step is set
as 
  \begin{equation}
    \label{eq:dt}
    \Delta t =  C \Delta x / c.    
  \end{equation}
  In Figure \ref{fig:2B}, the numerical results of the radiation wave
  front at time $t = 10, 50$ and $100$ are plotted.  In Figure
  \ref{fig:2B_1}, the reference is obtained by the $S_N$ method, and
  in Figure \ref{fig:2B_2}, the reference solution is from
  \cite{semi2008Ryan} and the diffusion limit result is produced by
  the finite difference method. From Figure \ref{fig:2B}, we can find
  that the numerical solution to RTE is on top of each other with that
  of the reference solution and the diffusion limit results, which is
  also consistent with the expectation that the solution to RTE is
  almost the same as the diffusion limit.

\paragraph{Marshak Wave-2A}
Marshak Wave-2A problem is quite similar to Marshak Wave-2B problem,
except that its absorption/emission coefficient is
$\sigma = 10 {\rm /T^3 cm^2/g}$. In this case, since $\sigma$ is not
large enough, the solution to RTE is different from that of the
diffusion limit.

In this test, the same numerical setting as Marshak wave-2B problem is
chosen.  In Figure \ref{fig:2A_1}, the computed radiation wave front
at time $t = 0.2, 0.4, 0.6, 0.8$ and $1.0$ are given and the reference
is obtained from the $S_N$ method.  Figure \ref{fig:2A_2} presents the
computed material temperature for both the gray approximation to the
radiation transfer equations and the nonlinear diffusion equation at
time $t = 1$, where the reference solution is from
\cite{mila2000}. From it, we can see that the numerical solution
matches the reference solution well, but is quite different from the
diffusion limit.

\begin{figure}[!htb]
  \centering
  \subfloat[material temperature]{
    \label{fig:2A_1}
    \includegraphics[width=0.4\textwidth]{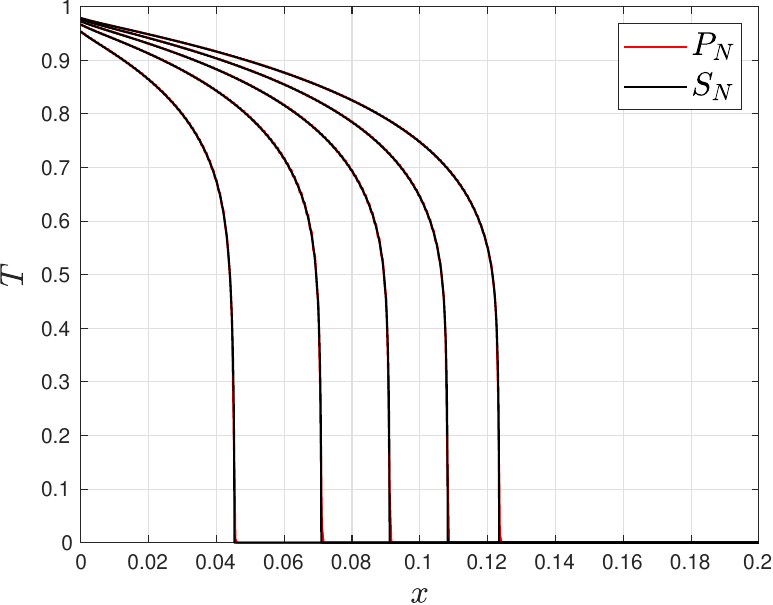}}
  \hfill
  \subfloat[material temperature]{
    \label{fig:2A_2}
    \includegraphics[width=0.4\textwidth]{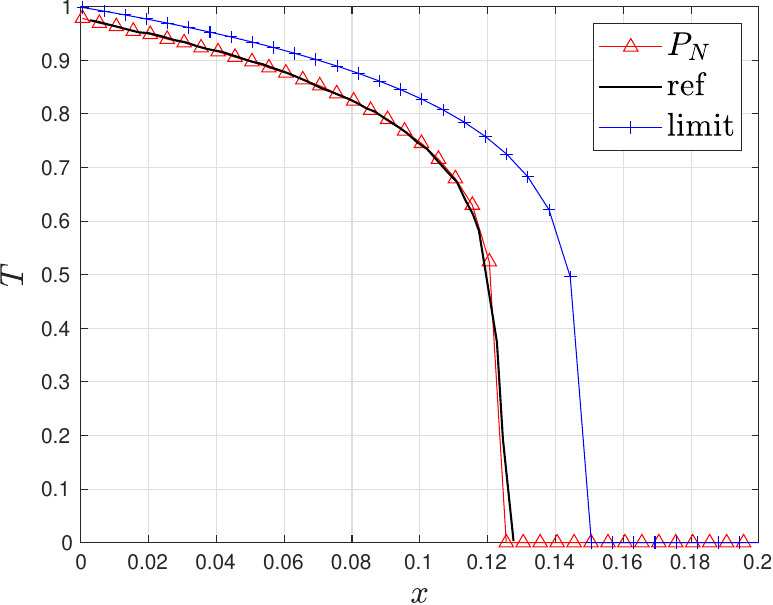}}
  \caption{The material temperature $T$ of Marshak Wave-2A problem at
      different time. (a) The material temperature $T$ of Marshak Wave-2A
      problem at $t = 0.2, 0.4, 0.6, 0.8$ and $1$. The black 
      line is the reference solution obtained by $S_N$ method. (b) The
      material temperature $T$ of Marshak Wave-2A problem at $t = 1$, where
      the red line is the numerical solution to RTE, the black line is
      the reference solution and the blue line is that to diffusion
      limit. }
\end{figure}

From the numerical results of Marshak wave problems, we can find that
the new numerical scheme works well both for the optically thick and
thin problems. The time step length is independent of the absorption
coefficients $\sigma$, which shows the high efficiency of this AP
numerical scheme.

\subsection{A lattice problem}
\begin{figure}[!htb]
  \centering
  \includegraphics[width=0.4\textwidth]{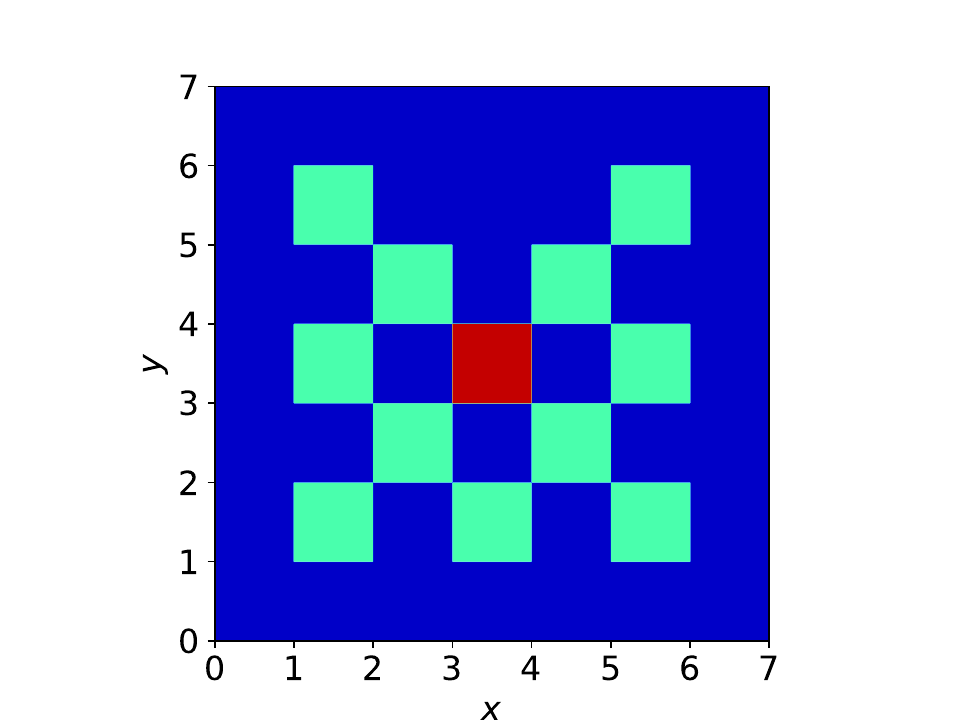}
  \caption{Layout of the lattice problem. }
  \label{fig:lattice_mesh}
\end{figure}

In this section, we study a two-dimensional problem with the added
complication of multiple materials but without radiation-material
coupling. We consider the transfer equation 
\begin{equation}
  \dfrac{\epsilon^2}{c}\pd{I}{t} + \epsilon \bsOmega \cdot \nabla I =
- \sigma_a I + \sigma_s \left(\dfrac{1}{4\pi}\int_{\bbS^2} I \dd
\bsOmega - I\right) + \epsilon^2 G.  
\end{equation} 
Photons are absorbed with a power density of
$\dfrac{c \sigma_a}{\epsilon^2} I$.  As there is no radiation-material
coupling, the photons are simply removed from the system when the
absorption occurs. The isotropic scattering term does not change the
radiation temperature but causes the specific intensity $I$ to become
more evenly distributed in microscopic velocity. The computation
domain is $[0, 7] \times [0, 7]$. It consists of a set of squares
belonging to a strongly absorbing medium in a background of weakly
scattering medium.  The specific layout of the problem is given in
Figure \ref{fig:lattice_mesh}, where the blue regions and the dark red
region are purely scattering medium with $\sigma_s = 1$ and
$\sigma_a = 0$; the light green regions contain purely absorbing
material with $\sigma_s = 0$ and $\sigma_a = 10$. In the dark red
region, there is an isotropic source $G = \dfrac{1}{4\pi}$, and $G$ is
zero elsewhere.  Initially the specific intensity is at equilibrium
and the radiation temperature is $10^{-6}$. Vacuum boundary conditions
are imposed on all four sides of the computation domain, which means
there is an outflow of radiation but no inflow. The detailed
  application of the inflow boundary is proposed in Appendix
  \ref{app:boundary_condition}, and we also refer to \cite{semi2008Ryan}
  for more details.  Other parameters are set as
$c = a = \epsilon = 1$.

We use a mesh of $280 \times 280$ in the spatial space and $P_5$ is
adopted here. Moreover, the filtering technique is applied in the
microscopic velocity space to avoid negative energy density solution.
Filtering techniques as proposed in \cite{mcclarren2010robust} are
employed in 2D simulations to suppress spurious oscillations in $P_N$
solutions. The filtering applied here is only applied to
$l \geqslant 2 M / 3$, where $M$ is the highest order of spherical
harmonic expansion \cite{hou2007computing, Filter2017}. For these $l$,
before updating  each time step, we substitute $I^m_l$ with
$\hat{I}^m_l$. Precisely
\begin{equation}
  \hat{I}^m_l = \dfrac{I^m_l}{1 + \alpha l^2 (l+1)^2},
\end{equation}
where
\begin{equation}\label{eq:alpha}
  \alpha = \dfrac{\omega}{M^2}\dfrac{1}{[(\sigma_a + \sigma_s)L +
    M]^2},\qquad
  \omega = \dfrac{2 c \Delta t}{\Delta x},
\end{equation}
with $L$ the characteristic length of the problem, which is taken to
be $L=1$ for all the simulations. $\sigma_a$ and $\sigma_s$ are the
absorption and scattering coefficients, respectively. 

\begin{figure}[htbp]
  \centering
  \subfloat[$\log_{10}I^0_0$ by AP IMEX.]{
    \label{fig:lattice_ap}
    \includegraphics[width=0.3\textwidth]{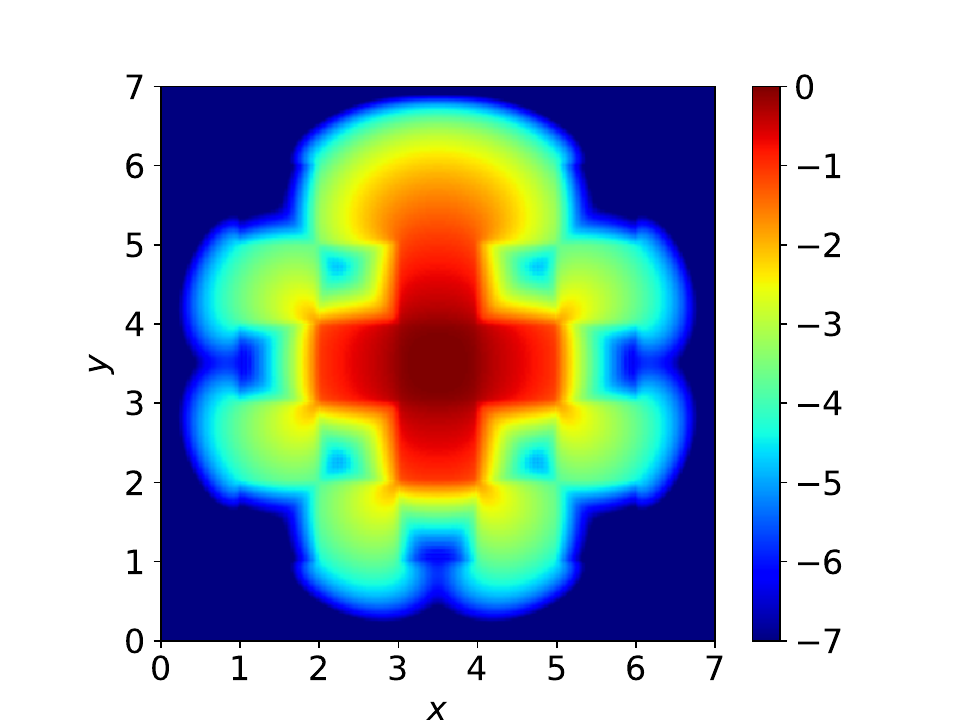}}
  \hfill
  \subfloat[$\log_{10}I^0_0$ by StaRMAP.]{
    \label{fig:lattice_starmap}
    \includegraphics[width=0.3\textwidth]{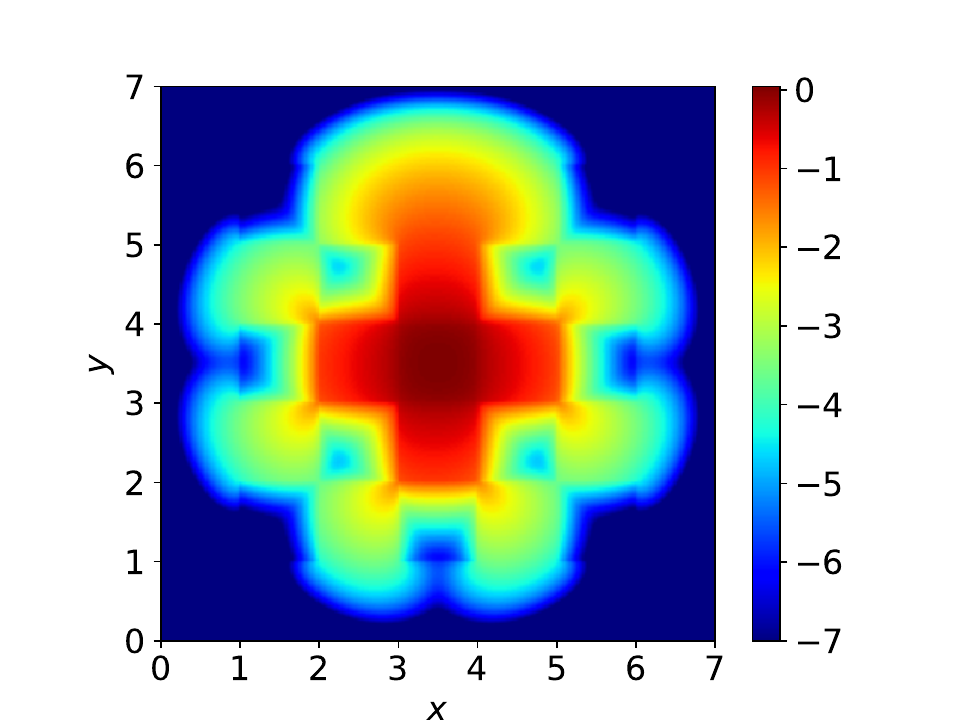}}
  \hfill \subfloat[$\log_{10}I^0_0$ at $x=3.5$]{
    \label{fig:lattice_cut}
    \includegraphics[width=0.30\textwidth]{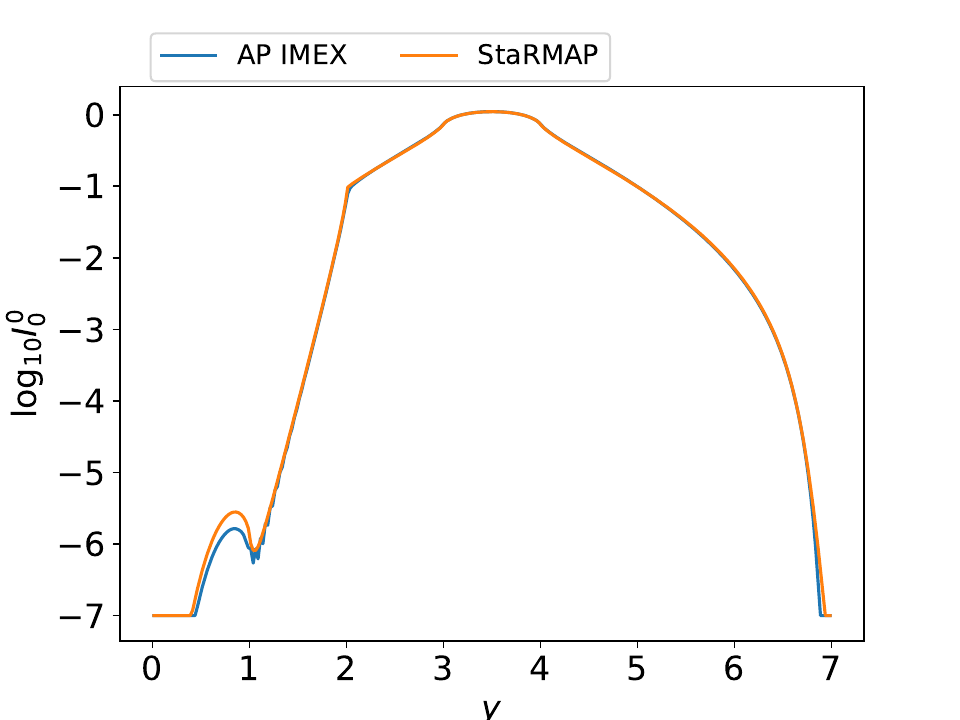}}
  \caption{The contour and slice plot of $\log_{10}I^0_0$ for the
    lattice problem at $t = 3.2$. (a) Contour plot of the AP numerical
    scheme. (b) The contour plot of the reference solution reproduced
    using StaRMAP.  (c) Comparison between the numerical solution and
    the reference solution reproduced by StaRMAP at $x = 3.5$.}
    \label{fig:lattice}
\end{figure}

In the test, the first-order scheme \eqref{eq:first_order} is utilized
for temporal discretization and a third-order WENO reconstruction is
adopted in spatial discretization. The results at time $t = 3.2$ are
shown in Figure \ref{fig:lattice} with the logarithm of $I_0^0$ to the
base 10 shown in contour and slice. The reference solution is obtained
using StarRMAP \cite{seibold2012starmap, seibold2014starmap}. Figure
\ref{fig:lattice} show that both solutions agree with each other quite
well, and the beams of the particles leaking between the corners of
the absorbing regions are all well produced. This phenomenon is also
studied in \cite{brunner2002forms, seibold2014starmap}, and the
behavior of the numerical results are almost the same.

\subsection{A hohlraum problem}
\begin{figure}[!htb]
  \centering
  \includegraphics[width=0.35\textwidth]{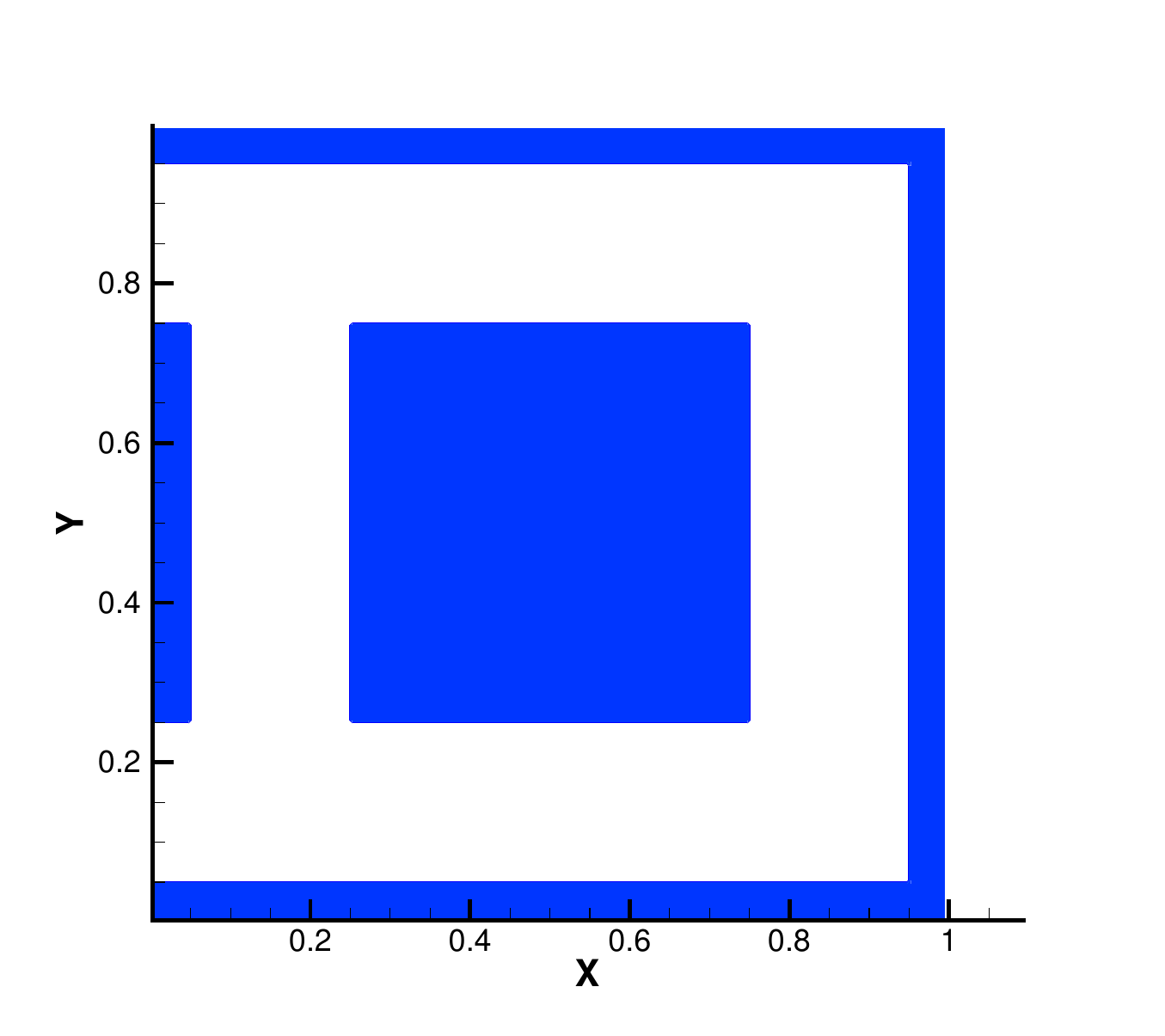}
  \caption{Layout of the hohlraum problem. The blue regions are where
  $(x,y) \in [0, 0.05] \times [0.25,0.75]$, and $(x,y) \in [0.25,
  0.75]\times[0.25, 0.75]$, $(x,y) \in [0,1]\times[0,0.05]$, $(x,y)
  \in[0,1]\times[0.95,1]$ and $(x,y) \in [0.95,1] \times [0,1]$.}
  \label{fig:hohlraum_coeff}
\end{figure}

This section studies the hohlraum problem, which is similar to that in
\cite{mcclarren2010robust}. For this problem, the radiation field is
coupled with the material energy. It is well known that the diffusion
approximation fails to capture the correct physics of this problem
\cite{brunner2002forms, mcclarren2010robust}, making it necessary to
simulate the original RTE \eqref{eq:RTE}.  Moreover, in this problem,
the material is initially cold and optically thick, and then becomes
optically thinner as radiation heats it up. The wide range in optical
depth presents a challenge to the numerical schemes. The layout of the
problem is shown in \Cref{fig:hohlraum_coeff}.  The computation domain
is $[0,1]\times[0,1]$, where the white areas are vacuum with
$\sigma_a = 0$. The blue regions in \Cref{fig:hohlraum_coeff} satisfy
$\sigma_a = 100 {\rm /T^3 cm^2/g}$, while the density is
$1.0 \rm {g/cm^3}$ and the heat capacity $C_v$ is
$0.3 \rm {GJ/g/keV}$.  An isotropic inflow of $1$ keV black body
source is incident on the entire left boundary. For the boundary
conditions, the Marshak type inflow boundary condition is applied. For
the left boundary there is an isotropic inflow, and for other
boundaries, the outside is treated as vacuum. Therefore, there is an
outflow of radiation but no inflow in other three boundaries. The
details of the Marshak type inflow are also proposed in Appendix
\ref{app:boundary_condition}. In the simulation, the related
parameters are set as $\epsilon = 1$,
$a = 0.01372 \rm{GJ/cm^3-ke V^4}$ and $c = 29.98 \rm{cm / ns}$. The
mesh size is $100 \times 100$ in the spatial space and the $P_N$
method with $M = 7$ is utilized. The first-order scheme for the time
discretization and third-order WENO reconstruction in the spatial
discretization is utilized here with the same filtering techniques in
the last section.

\begin{figure}[!htb]
  \centering
  \subfloat[Radiation temperature]{
    \label{fig:hohlraum_imc}
    \includegraphics[width=0.48\textwidth]{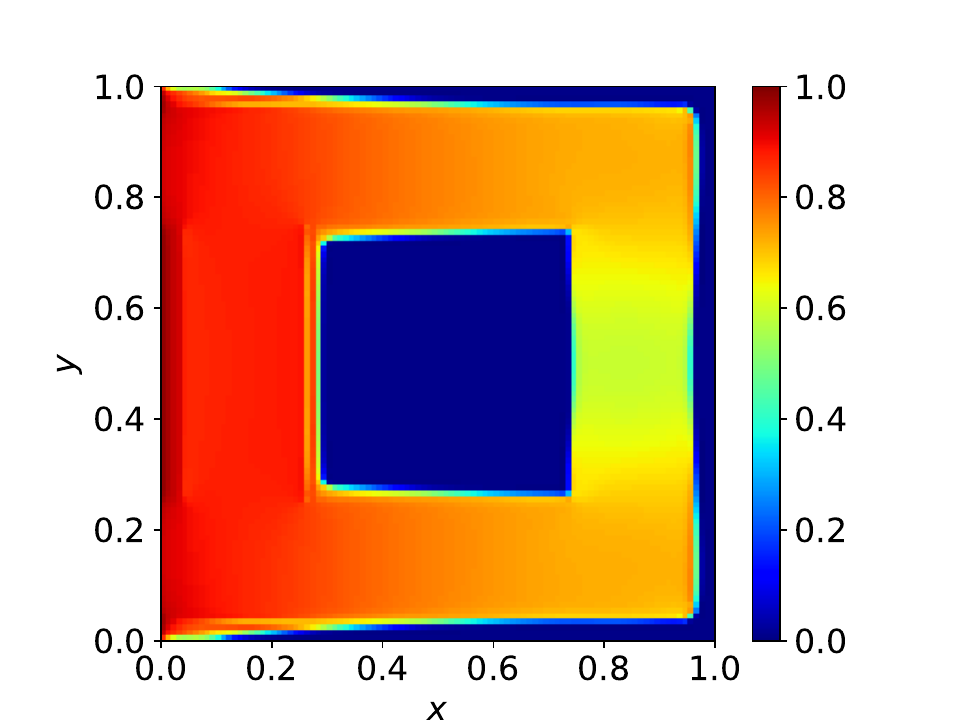}}
  \hfill
  \subfloat[Material temperature]{
    \label{fig:hohlraum_Pn}
    \includegraphics[width=0.48\textwidth]{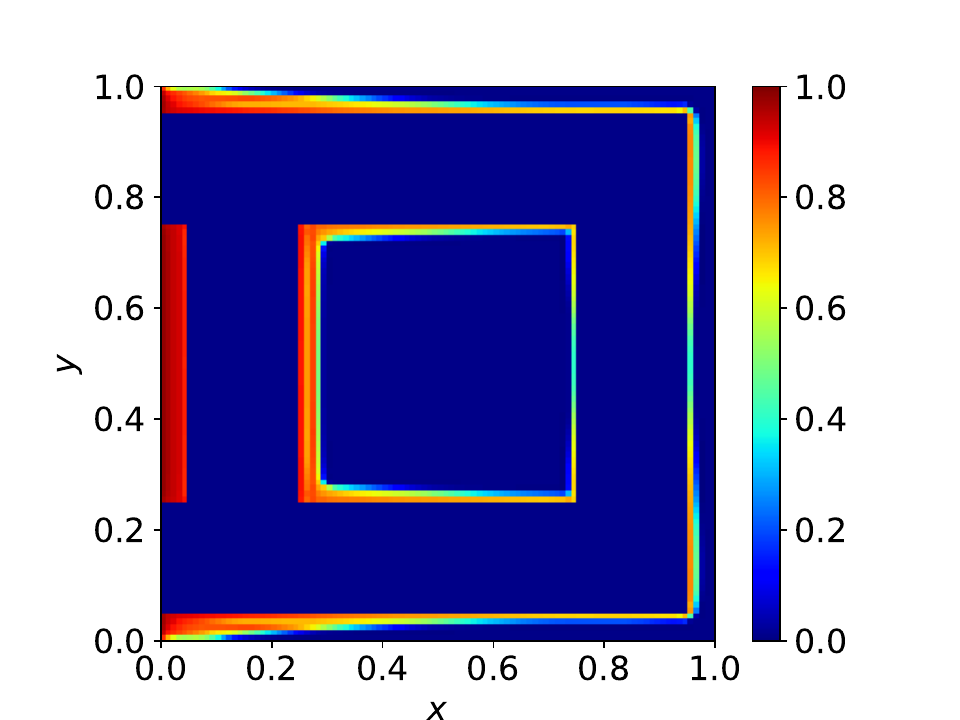}}
  \caption{The contour plots of radiation and material temperature of
    the hohlraum problem produced by the AP scheme at $t = 1$. (a)
    Radiation temperature. (b) Material temperature.}
    \label{fig:hohlraum_contour}
  \end{figure}

  \cref{fig:hohlraum_contour} presents the contour plots of the
  numerical solution for the radiation temperature and the material
  temperature at $t = 1$, where the radiation temperature is defined
  as
\begin{equation}
  T_{\rm rad} = \sqrt[4]{\dfrac{I^0_0}{a c}}.
\end{equation}
As is stated in \cite{mcclarren2010robust} that the solution to this
problem has two properties, first of which is the non-uniform heating
of the central block, and the other is less radiation directly behind
the block than those regions within the line of source sight. The same
phenomenon could also be found in the numerical results
here. Moreover, the numerical results also show that the photons could
bend around the front wall and the back wall is starting to heat up
and re-emit photons. The numerical solutions along $y = 0.125$ and
$x = 0.85$ are plotted in Figure \ref{fig:hohlraum_slice}, where the
solution obtained by IMC method \cite{mcclarren2010robust} is also
plotted. We can find that the numerical solutions are in rough
agreement with the IMC solution.

\begin{figure}[!htb]
  \centering
  \subfloat[$y=0.125$]{
    \label{fig:hohlraum_slice1}
    \includegraphics[width=0.48\textwidth]{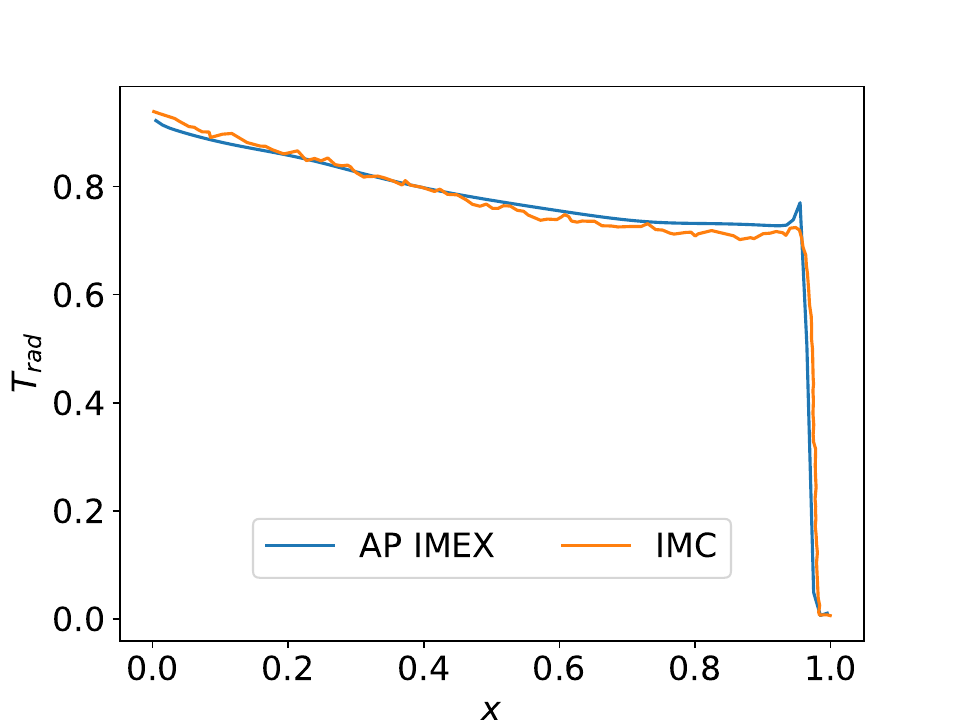}}
  \hfill
  \subfloat[$x = 0.85$]{
    \label{fig:hohlraum_slice2}
    \includegraphics[width=0.48\textwidth]{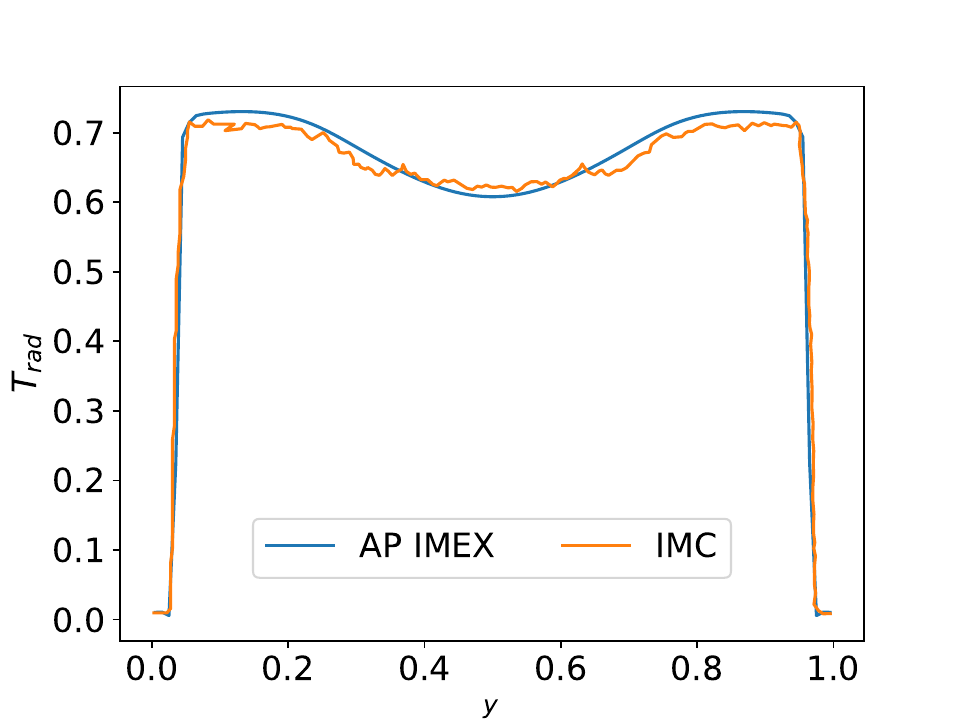}}
  \caption{The radiation temperature of the hohlraum problem at $t =1$
    on different slices. Here, the blue line the numerical solution by
    the AP numerical method, and the red line is the reference
    solution by IMC method in \cite{mcclarren2010robust}. (a)
    radiation temperature $T_{\rm rad}$ at $y = 0.125$. (b) radiation
    temperature $T_{\rm rad}$ at $x = 0.85$.}
    \label{fig:hohlraum_slice}
\end{figure}


\section{Conclusions}
\label{sec:conclusion}
In this paper, we have developed an AP IMEX numerical scheme for the
RTE system in the framework of $P_N$ method. The Chapman-Enskog
expansion is utilized to derive the order of each expansion
coefficient of the specific intensity respected to the mean free
path. Thus, in each equation of the $P_N$ system, the terms at
lower-order of the mean free path are set as an implicit term with
those at higher-order set as an explicit term.  Therefore, the
implicit-explicit $P_N$ system can be solved at the computational cost
of a completely explicit scheme with the time step length independent
of the mean free path. The analysis of the total energy shows the
energy stability with the evolution of time. Numerical examples have
exhibited the AP property and the efficiency of this new
scheme. However, this method is limited to the gray approximation of
the radiative transfer equations for the moment. Research works on the
frequency-dependent problem are ongoing.

\section*{Acknowledgements}

We thank Prof. Ruo Li from PKU, Prof. Zhenning Cai from NUS, Prof. Tao
Xiong from XMU, Prof. Kailiang Wu from SUSTech, Dr. Zhichao Peng from
MSU and Prof. Jiequan Li, Prof. Wenjun Sun, Dr. Yi Shi from IAPCM for
their valuable suggestions.  Weiming Li is partially supported by the
Science Challenge Project (No. TZ2016002) and the National Natural
Science Foundation of China (12001051). Peng Song is partially
supported by the Science Challenge Project (No.  TZ2016002), the CAEP
foundation (No. CX20200026).  The work of Yanli Wang is partially
supported by Science Challenge Project (No.  TZ2016002) and the
National Natural Science Foundation of China (Grant No. 12171026,
U1930402 and 12031013).

\appendix

\section{Appendix}
\label{sec:app}
\subsection{The gray approximation of the radiative transfer equations
  for 1D angle problem and related $P_N$ equations}
\label{app:1D}
The time-dependent gray approximation of the radiative transfer equations \cite{semi2008Ryan} in a
one-dimensional planar geometry medium have the form as
\begin{subequations}
  \label{eq:1D_RTE}
  \begin{align}
    \label{eq:1D_RTE_I}
    & \frac{\epsilon^2}{c} \pd{I}{t} +\epsilon \mu \pd{I}{x} =
    \sigma\left(\frac{1}{2}a c   T^4 - I\right),  \qquad x \in [0, L], \\
    \label{eq:1D_RTE_T}
    & \epsilon^2 C_{v} \pd{T}{t} = \sigma \left(\int_{-1}^1 I \dd \mu - ac
      T^4\right),
  \end{align}
\end{subequations}
where $I(t,  x, \mu)$ is the specific intensity of radiation,
$\mu = \cos \theta \in [-1, 1]$ is the internal coordinate associated
with the angle $\theta \in [0, \pi]$. $T(t, x)$ is the material
temperature, and $\sigma$ is the absorption opacity. Moreover, the
one-dimensional form of \eqref{eq:linear_RTE} is reduced into
\begin{equation}
  \label{eq:plane_eq}
  \frac{\epsilon^2}{c} \pd{I}{t} +\epsilon \mu \pd{I}{x} =
  \sigma\left(\frac{1}{2} \int I \dd \mu - I\right),  \qquad x \in [0, L]. \\
\end{equation}
For \eqref{eq:1D_RTE}, the basis function for the $P_N$ method is the
Legendre polynomials. The moments are defined as
\begin{equation}
  \label{eq:1D_moment}
  I_l =  \int_{-1}^1 P_l(\mu) I(t, x, \mu) \dd \mu, \qquad l = 0,
  \cdots M,
\end{equation}
where $P_l$ is the Legendre polynomial.  Then, the $P_N$ equations for
\eqref{eq:1D_RTE} are
\begin{equation}
  \label{eq:1D_Pn}
  \begin{aligned}
    & \frac{\epsilon^2}{c}\pd{\bbI}{t} + \epsilon\bB^{\rm low }
    \pd{\bbI}{x} + \epsilon {\bB}^{\rm up} \pd{{\bbI}}{x} = -\sigma
    {\bbI} + \sigma a c
    T^4  e_1, \\
    & \epsilon^2 C_{v} \pd{T}{t} = \sigma \left(I_0 - ac T^4\right),
  \end{aligned} 
\end{equation}
where $\bbI = (I_0, I_1, \cdots, I_M)$ and
$e_1 = (1, 0, \cdots, 0)^T$.  $\bB^{\rm low}$ and $\bB^{\rm up}$ are
triangular matrix with the non-zero entries as
\begin{equation}
  \label{eq:B}
  \bB^{\rm low}(i+1, i) = \frac{i}{2 i+1}, \qquad \bB^{\rm up}(i,
  i+1) = \frac{i}{2(i-1)+1}, \qquad i = 1, \cdots M.
\end{equation}

\subsection{Inflow boundary condition for $P_N$ equations}
\label{app:boundary_condition}
We implement the inflow boundary condition for the $P_N$ equations by
specifying the values of coefficients of the $P_N$ system in ghost
cells. Choosing the left boundary as an
example, the incoming specific intensity incident on the boundary
interface is
\begin{equation}
  I(\mu) = I^b(\mu), \quad \text{for}~\mu > 0.
\end{equation}
For the $P_N$ method, the numerical boundary can be rewritten as
\begin{equation}
  \label{eq:boundary_1D}
  I^{\rm ghost}(\mu) = \left\{
    \begin{array}{cc}
      I^{b}(\mu), &\mu > 0, \\
      I^{i}(\mu), & \mu < 0,
    \end{array}
  \right.
\end{equation}
where $I^{i}(\mu)$ is the specific intensity at the left boundary of
the area.
Then, the expansion coefficient at the ghost cell is
\begin{equation}
  I^{\rm ghost}_{l} = \int_{-1}^1 I^{\rm ghost}(\mu) P_l(\mu)\dd \mu.
\end{equation}



The implementation of the inflow boundary condition in 2D is similar
in spirit to that of 1D. Supposing $\bn$ is the outward normal of the
boundary interface, the incident specific intensity on the boundary is
\begin{equation}
  I(\bsOmega) = I^b(\bsOmega), \quad \text{for}~\bsOmega \cdot \bn <
  0.
\end{equation}
For the $P_N$ method, the numerical boundary can be rewritten as
\begin{equation}
  \label{eq:boundary_2D}
  I^{\rm ghost}(\bsOmega) = \left\{
    \begin{array}{cc}
      I^{b}(\bsOmega), & \bsOmega \cdot \bn < 0, \\
      I^{i}(\bsOmega), & \bsOmega \cdot \bn > 0,
    \end{array}
  \right.
\end{equation}
where $I^{i}(\bsOmega)$ is the specific intensity on the interior
side of the boundary interface.
Thus, the expansion coefficient at the ghost cell is
\begin{equation}
  \begin{aligned}
    I^{m, \rm ghost}_{l} & = \int_{\bbS^2} I^{\rm ghost}(\bsOmega)
    \overline{Y^m_l}(\bsOmega) \dd\bsOmega \\
    & = \int_{\bsOmega \cdot \bn < 0} I^{b}(\bsOmega)
    \overline{Y^m_l}(\bsOmega) \dd \bsOmega + \int_{\bsOmega \cdot \bn
      > 0} I^{i}(\bsOmega)
    \overline{Y^m_l}(\bsOmega) \dd \bsOmega \\
    & = \int_{\bsOmega \cdot \bn < 0} I^{b}(\bsOmega)
    \overline{Y^m_l}(\bsOmega) \dd \bsOmega +
    \sum\limits_{j=0}^M\sum\limits_{k=-j}^j I^{k, i}_j \int_{\bsOmega
      \cdot \bn > 0} Y^k_j(\bsOmega) \overline{Y^m_l}(\bsOmega) \dd
    \bsOmega.
  \end{aligned}
\end{equation}
The integration $\int_{\bsOmega \cdot \bn > 0} Y^k_j(\bsOmega)
\overline{Y^m_l}(\bsOmega) \dd \bsOmega$ does not depend on the specific numerical solutions, and is pre-computed.

\subsection{Proof of Proposition \ref{thm:fourier}}
\label{app:fourier}

In this section, the proof of the Proposition \ref{thm:fourier} is proposed here.

{\renewcommand\proofname{Proof of Proposition \ref{thm:fourier}}
  \begin{proof}
    Following the method in \cite{Peng2020}, we will begin the Fourier
    analysis of \eqref{eq:first_order} for the $P_1$ system of the
    linear equation system \eqref{eq:plane_eq}.  The result can be
    extended to the generalized $P_N$ system naturally. We will
first discuss two special cases where $\xi = 0, \pi$. Therein, $\bf C$
is reduced into a real diagonal matrix with the maximum eigenvalues
equaling $1$.  According to the principle, the numerical scheme is
stable.

Then, we study the general case by considering two scenarios according
to the time step length:
\begin{enumerate}
  \item $\epsilon < \Delta x$, where
  \begin{equation}
    \label{eq:fourier_time_2}
    \Delta t =  C \Delta x^2.    
  \end{equation}
  Substituting the time step length \eqref{eq:fourier_time_2} into
  \eqref{eq:fourier_g}, we can find that $\lambda_i, i=1,2$ are
  functions of $C$, $\frac{\epsilon}{\Delta x}$, $\alpha$ and
  $\xi$. Introducing two variables as $\beta_1 = \log_{10}(C)$ and
  $\beta_2 = \log_{10} (\epsilon / \Delta x)$, the
  stability regions are plotted in Figure \ref{fig:Fourier_1} with fixed
  $\alpha$. Here the discrete wave number $\xi$ is uniformly taken
  from $[0, 2\pi]$ with $200$ samples. In this case, due to the
  definition of $C$ which is the CFL number and $\epsilon < \Delta x$,
  the range for $\beta_1$ and $\beta_2$ is changed into
    \begin{equation}
      \beta_1 < 0, \qquad \beta_2 < 0.
    \end{equation}
    Moreover, it is natural to demand that
    $\epsilon < \Delta x < 0.4$. Thus, $\alpha$ is taken uniformly
    from $[0, \exp(-1/0.16)]$ with 100 samples and six cases are shown
    in Figure \ref{fig:Fourier_1} due to their similar behavior. From
    Figure \ref{fig:Fourier_1}, we can find that when $\alpha = 0$,
    the numerical scheme is always stable.  However, with the
    increase of $\alpha$, the stability region is becoming smaller,
    especially when the CFL number $C$ is large and the radio
    $\epsilon / \Delta x$ is small. We find that when
    $\alpha = \exp(-1/0.16)$, the numerical scheme is stable when
    $\log_{10}(\epsilon / \Delta x) > -2.5$. Noting that when
    $\alpha = \exp(-1/0.16)$, which means $\epsilon = 0.4$,
    $\log_{10}(\epsilon / \Delta x)$ is always larger than $-2.5$ for
    $\Delta x < 1$. This indicates that in the simulation of benchmark
    problems, the stability condition is always satisfied.

\item $\epsilon > \Delta x$, where
  \begin{equation}
    \label{eq:fourier_time_1}
    \Delta t = C \epsilon \Delta x.  
  \end{equation}
  Substituting the time step length \eqref{eq:fourier_time_1} into
  \eqref{eq:fourier_g}, we can easily find that $\lambda_i, i = 1, 2$
  are also the function of $C$, $\frac{\epsilon}{\Delta x}$, $\alpha$
  and $\xi$. Introducing the same two variables $\beta_{i}, i = 1,2$,
  we plot the stability regions in Figure \ref{fig:Fourier} with fixed
  $\alpha$. Here the discrete wave number $\xi$ is uniformly taken
  from $[0, 2\pi]$ with $200$ samples. Since
  \begin{equation}
    \label{eq:alpha1}
    \alpha = \exp\left(-\frac{1}{\epsilon^2}\right), 
  \end{equation}
  and assuming $\epsilon < 1$ in the numerical test, $\alpha$ is
  taken uniformly from $[0, 0.5]$ with 100 samples. As their behavior
  is similar, the six cases $\alpha = 0, 0.05, 0.1, 0. 2, 0.3$ and $0.5$
  are plotted here to illustrate the result. Moreover, due to the
  definition of $C$ which is the CFL number, and the condition that
  $\epsilon > \Delta x$, it holds that
  \begin{equation}
    \beta_1 < 0, \qquad \beta_2 > 0.
  \end{equation}
  From Figure \ref{fig:Fourier}, we can find that the numerical scheme
  is stable under the time step length \eqref{eq:fourier_time_1}. In
  the numerical tests, the upper bound of $\epsilon$ is
  $\epsilon = 10^5 \Delta x$, which is large enough for the
  computational parameter.
\end{enumerate}

\begin{figure}[!htb]
  \centering
  \subfloat[$\alpha = 0$]{
    \includegraphics[width=0.3\textwidth]{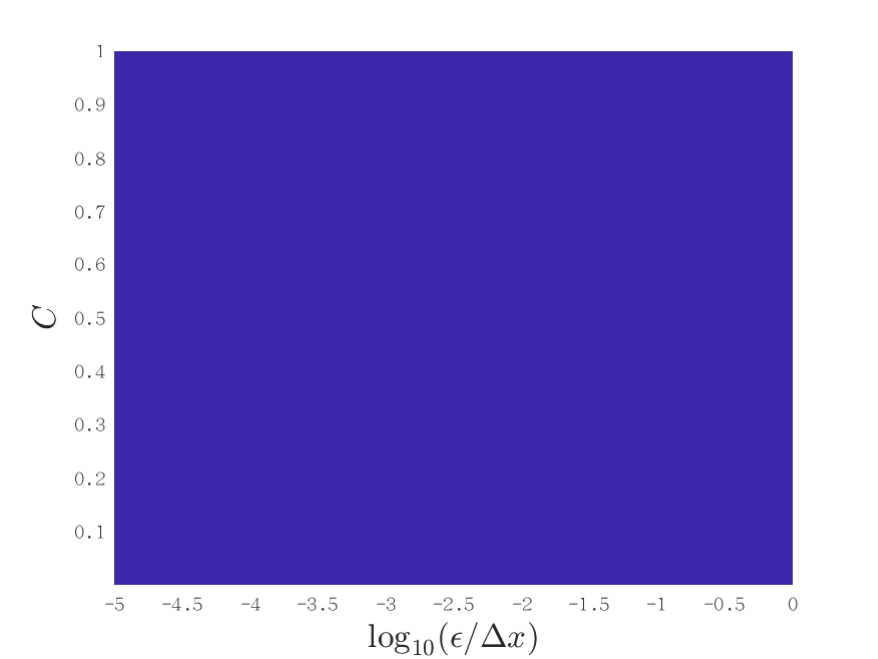}}
  \hfill
  \subfloat[$\alpha = 0.1 \exp(-1/0.16)$]{
    \includegraphics[width=0.3\textwidth]{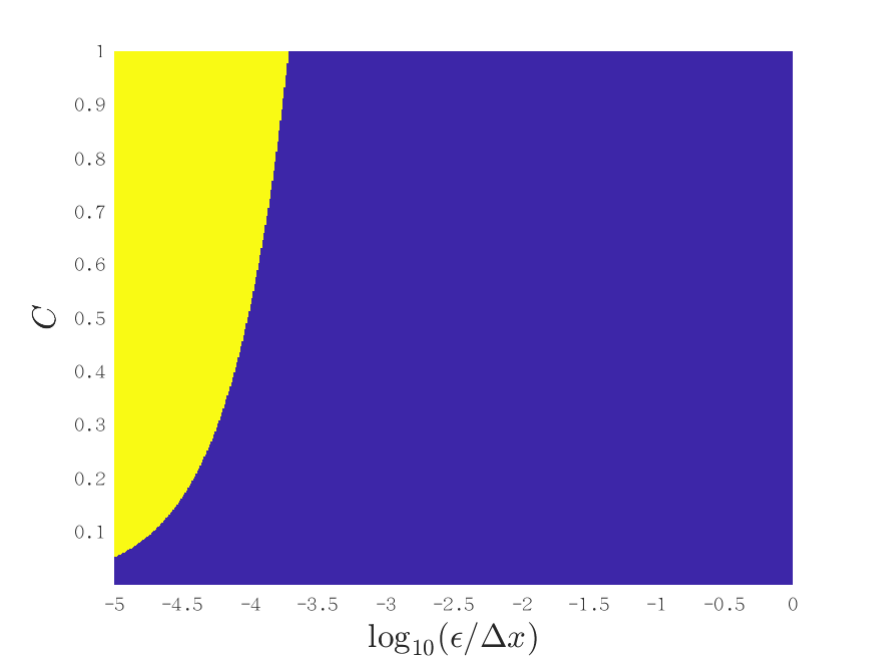}}
  \hfill
  \subfloat[$\alpha = 0.3\exp(-1/0.16)$]{
    \includegraphics[width=0.3\textwidth]{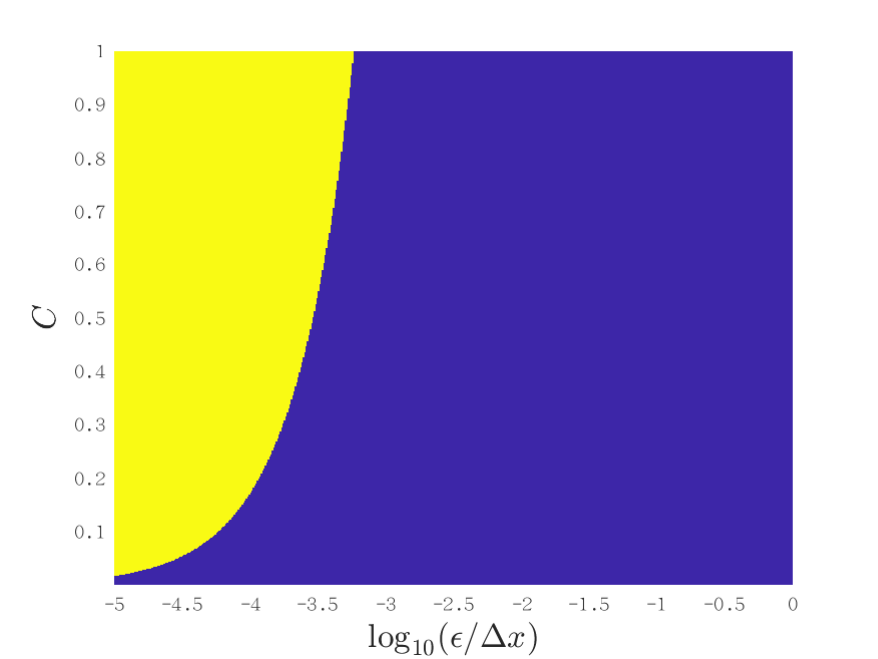}} \\
  \subfloat[$\alpha = 0. 5\exp(-1/0.16)$]{
    \includegraphics[width=0.3\textwidth]{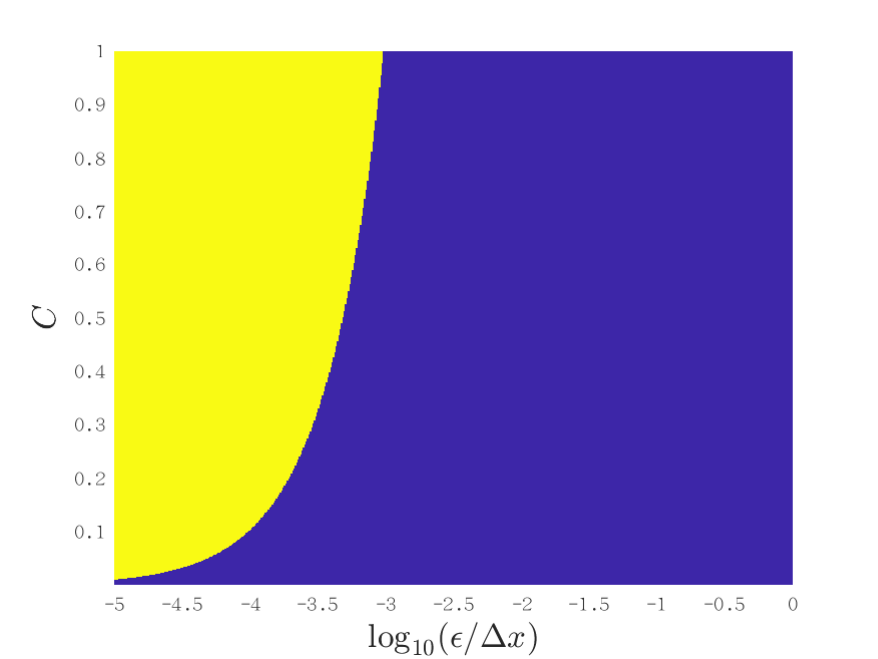}}
  \hfill
  \subfloat[$\alpha = 0.8\exp(-1/0.16)$]{
    \includegraphics[width=0.3\textwidth]{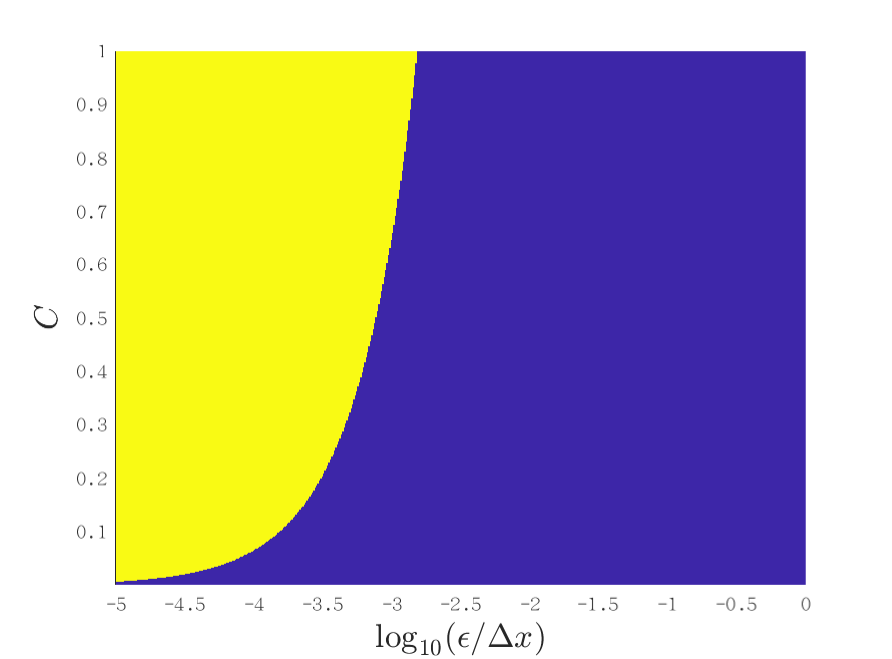}}
  \hfill
  \subfloat[$\alpha = \exp(-1/0.16)$]{
    \includegraphics[width=0.3\textwidth]{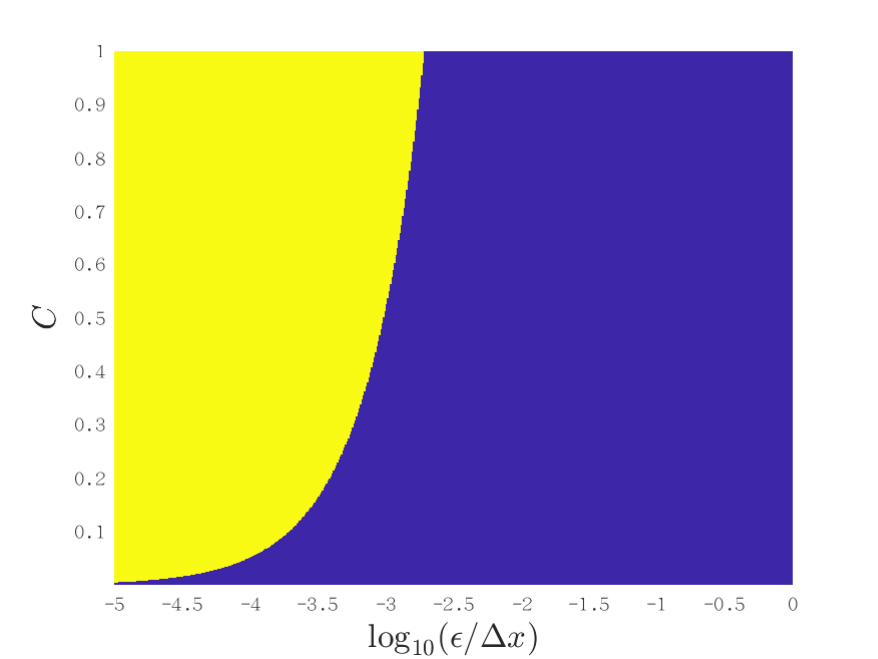}} \\
  \caption{The stability region for the numerical scheme
    \eqref{eq:first_order} of $P_1$ system under the condition
    \eqref{eq:fourier_time_2}. The $x$-axis is
    $\beta_2 = \log_{10}(\epsilon / \Delta x)$, and the $y$-axis is the
    CFL number $C$. The blue region is the area where the numerical
    scheme is stable and the yellow region is the area where the
    numerical scheme is unstable.}
  \label{fig:Fourier_1}
\end{figure}
\begin{figure}[!htb]
  \centering
  \subfloat[$\alpha = 0$]{
    \includegraphics[width=0.3\textwidth]{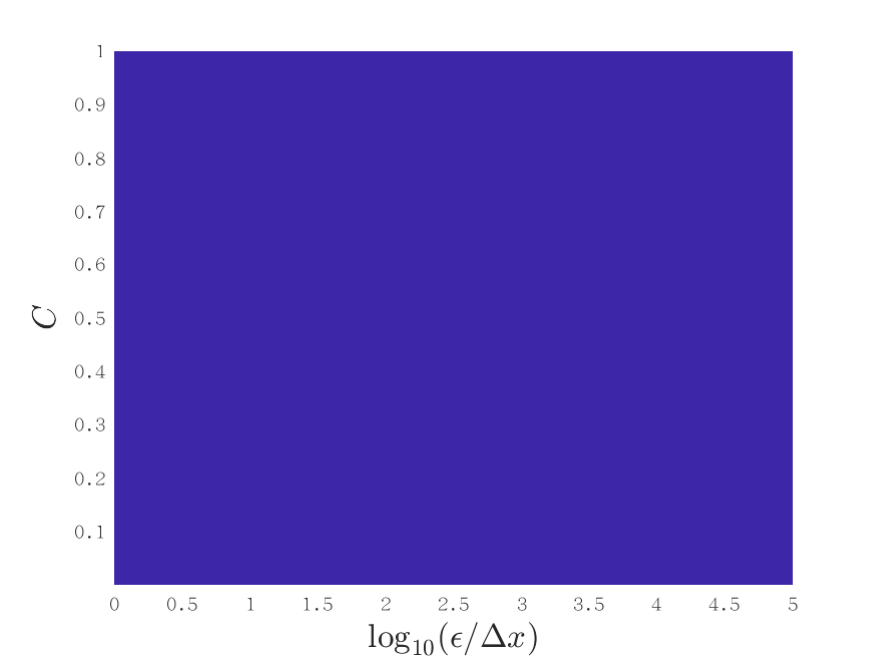}}
  \hfill
  \subfloat[$\alpha = 0.05$]{
    \includegraphics[width=0.3\textwidth]{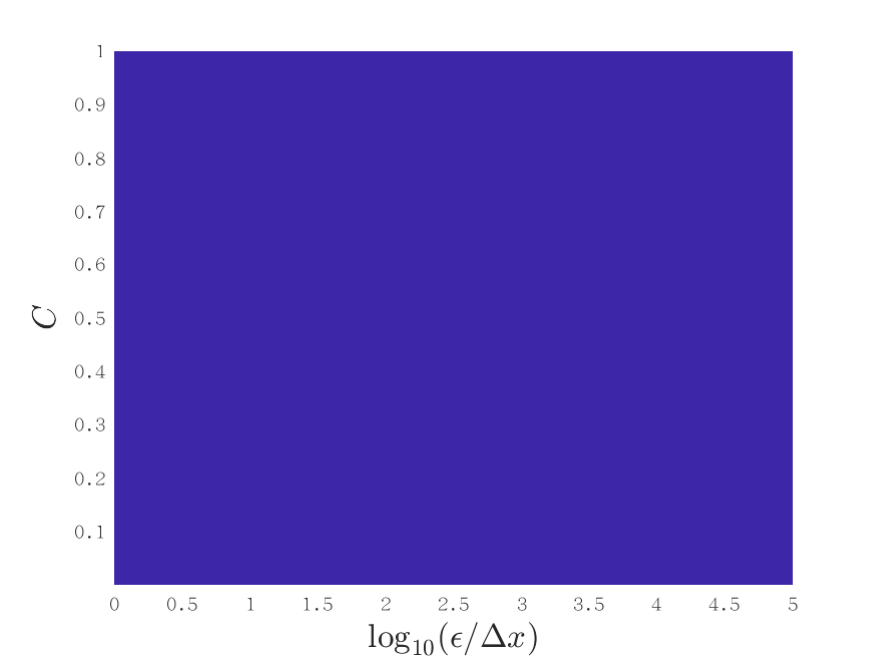}}
  \hfill
  \subfloat[$\alpha = 0.1$]{
    \includegraphics[width=0.3\textwidth]{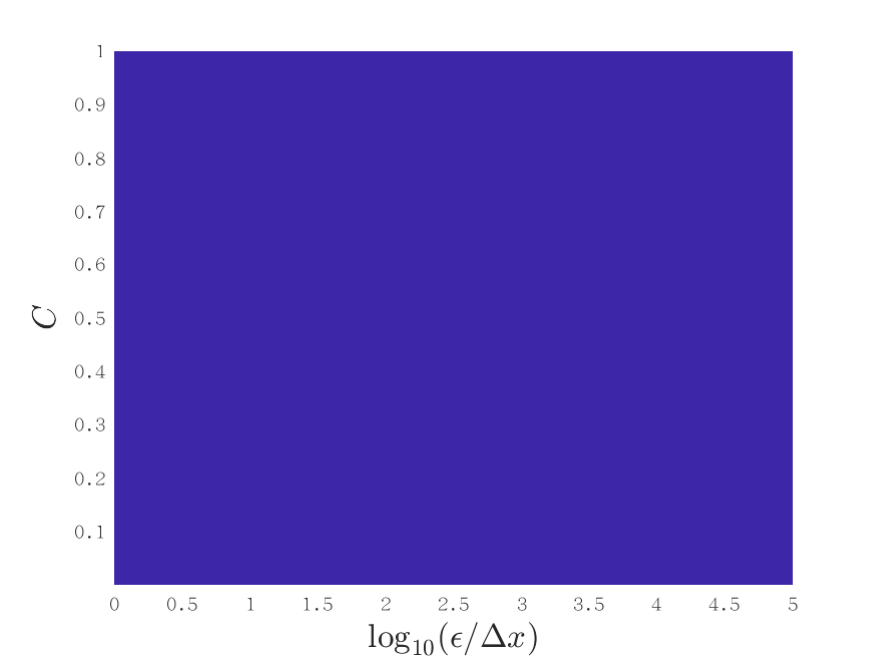}} \\
  \subfloat[$\alpha = 0. 2$]{
    \includegraphics[width=0.3\textwidth]{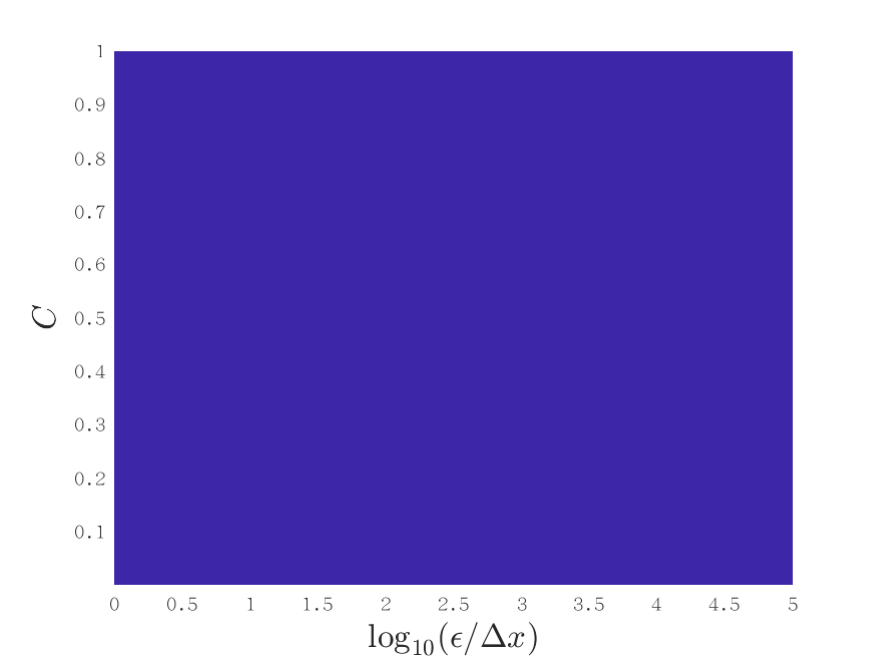}}
  \hfill
  \subfloat[$\alpha = 0.3$]{
    \includegraphics[width=0.3\textwidth]{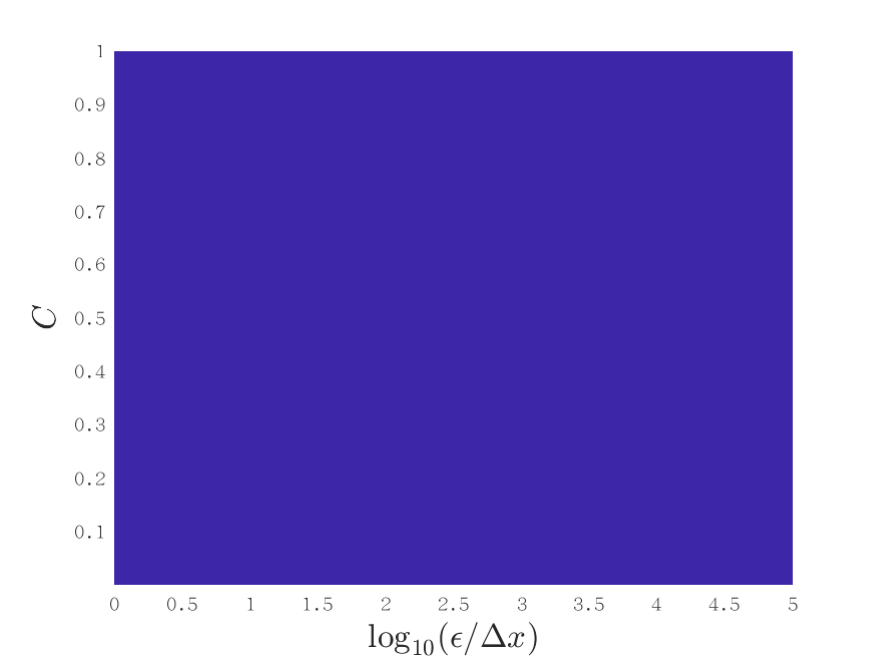}}
  \hfill
  \subfloat[$\alpha = 0.5$]{
    \includegraphics[width=0.3\textwidth]{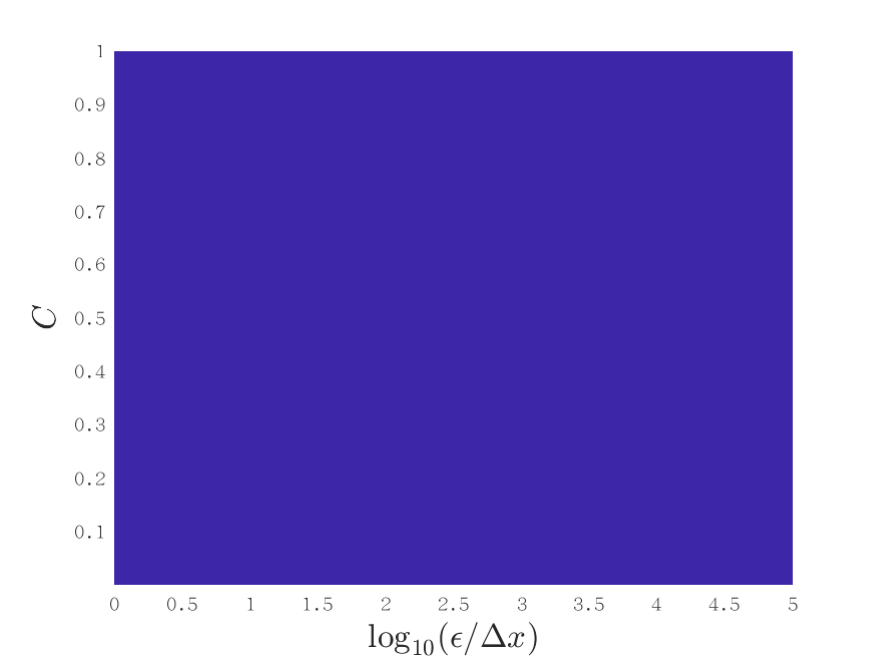}} \\
  \caption{The stability region for the numerical scheme
    \eqref{eq:first_order} of $P_1$ system under the condition
    \eqref{eq:fourier_time_1}. The $x$-axis is
    $\beta_2 = \log_{10}(\epsilon / \Delta x)$, and the $y$-axis is the
    CFL number $C$. The blue region is the area where the numerical
    scheme is stable.}
  \label{fig:Fourier}
\end{figure}
\end{proof}
}

\subsection{Proof of Theorem \ref{thm:discrete_energy}}
\label{app:energy_proof}
In this section, the proof of Theorem \ref{thm:discrete_energy} is
proposed here.

{\renewcommand\proofname{Proof of Theorem \ref{thm:discrete_energy}}
\begin{proof}
  We will take $M=2$ as an example, and it could be extended to the
  general case naturally. Moreover, without loss of generality, we set
  $a = c = C_v = \sigma = 1$ in the proof.  When $M=2$,
  \eqref{eq:1D_first_order} is reduced into
  \begin{subequations}
    \label{eq:P2}
    \begin{align}
      \label{eq:app_I0}
      \epsilon^2 \frac{I_{0,j}^{n+1} - I_{0,j}^n}{\Delta t}  & +
      \epsilon  \frac{I_{1, j+1}^{n} - I_{1, j-1}^{n}}{2 \Delta x} -
      \frac{\alpha \epsilon }{2} \frac{I_{0, j+1}^n - 2 I_{0,j}^n +
      I_{0, j-1}^n}{\Delta x}  = \left((T_j^4)^{n+1} -
      I_{0,j}^{n+1}\right), \\
      \label{eq:app_I1}
      \epsilon^2 \frac{I_{1,j}^{n+1} - I_{1,j}^n}{\Delta t}
                                                             & +
                                                               \frac{\epsilon}{3}  \frac{I_{0, j+1}^{n+1} - I_{0, j-1}^{n+1}}{2
                                                               \Delta x}
                                                               +\frac{2\epsilon}{3}  \frac{I_{2, j+1}^{n} - I_{2, j-1}^{n}}{2
                                                               \Delta x} - 
                                                               \frac{\alpha \epsilon }{2} \frac{I_{1, j+1}^n - 2 I_{1,j}^n +
                                                               I_{1, j-1}^n}{\Delta x}  = - I_{1,j}^{n+1}, \\
      \label{eq:app_I2}
      \epsilon^2 \frac{I_{2,j}^{n+1} - I_{2,j}^n}{\Delta t} & +
     \frac{2\epsilon}{5}  \frac{I_{1, j+1}^{n+1} - I_{1, j-1}^{n+1}}{2
      \Delta x} - 
      \frac{\alpha \epsilon }{2} \frac{I_{2, j+1}^n - 2 I_{2,j}^n +
      I_{2, j-1}^n}{\Delta x}  = - I_{2,j}^{n+1}, \\
      \label{eq:app_T}
      \epsilon^2 \frac{T_j^{n+1} - T_j^n}{\Delta t}
                                                             & +\epsilon \frac{I_{0,j}^{n+1} - I_{0,j}^n}{\Delta t} +
                                                               \epsilon  \frac{I_{1, j+1}^{n} - I_{1, j-1}^{n}}{2 \Delta x}
                                                               - \frac{\alpha \epsilon }{2} \frac{I_{0, j+1}^n - 2 I_{0,j}^n +
                                                               I_{0, j-1}^n}{\Delta x}
                                                               = 0.
    \end{align}
  \end{subequations}
  For \eqref{eq:app_I0}, multiplying it by $I_{0,j}^{n+1}$, we can get that
  \begin{equation}
    \label{eq:I0_1}
     \epsilon^2 I_{0,j}^{n+1} \frac{I_{0,j}^{n+1} - I_{0,j}^n}{\Delta t}   +
      \epsilon  I_{0,j}^{n+1}  \frac{I_{1, j+1}^{n} - I_{1, j-1}^{n}}{2 \Delta x} -
      \frac{\alpha \epsilon }{2} I_{0,j}^{n+1}  \frac{I_{0, j+1}^n - 2 I_{0,j}^n +
      I_{0, j-1}^n}{\Delta x}  = I_{0,j}^{n+1} \left((T_j^4)^{n+1} -
      I_{0,j}^{n+1}\right).
  \end{equation}
  For \eqref{eq:app_I1} and \eqref{eq:app_I2}, shifting it backward one time
  step and multiplying $3I_{1,j}^{n}$ and $5I_{2, j}^{n}$ respectively,
  we can derive that
  \begin{equation}
    \label{eq:I1_1}
    \begin{aligned}
      3\epsilon^2 I_{1,j}^n \frac{I_{1,j}^{n} - I_{1,j}^{n-1}}{\Delta
        t} & + \epsilon I_{1,j}^n \frac{I_{0, j+1}^{n} - I_{0,
          j-1}^{n}}{2 \Delta x} + \epsilon I_{1,j}^n \frac{I_{2,
          j+1}^{n-1} - I_{2, j-1}^{n-1}}{ \Delta x} - \frac{3\alpha
        \epsilon }{2} I_{1,j}^n\frac{I_{1, j+1}^{n-1} - 2
        I_{1,j}^{n-1} +
        I_{1, j-1}^{n-1}}{\Delta x}  = - 3(I_{1,j}^{n})^2, \\
      5 \epsilon^2 I_{2,j}^n\frac{I_{2,j}^{n} - I_{2,j}^{n-1}}{\Delta
        t} & + \epsilon I_{2,j}^n \frac{I_{1, j+1}^{n} - I_{1,
          j-1}^{n}}{ \Delta x} -\frac{5 \alpha \epsilon }{2}
      I_{2,j}^n \frac{I_{2, j+1}^{n-1} - 2 I_{2,j}^{n-1} +
        I_{2, j-1}^{n-1}}{\Delta x}  = - 5 (I_{2,j}^{n})^2.
    \end{aligned}
  \end{equation}
  Summing \eqref{eq:I0_1} and \eqref{eq:I1_1} over $j$, then it holds
  that
\begin{equation}
  \label{eq:sum_A3}
  \frac{\epsilon^2}{2\Delta t} \sum_j\Big[ (I_{0,j}^{n+1})^2 -
  (I_{0,j}^{n})^2 + 3\left((I_{1,j}^{n})^2 - (I_{1,j}^{n-1})^2
  \right)+ 5 \left((I_{2,j}^{n})^2 -
    (I_{2,j}^{n-1})^2\right) \Big]
  + A_0  = A_1 + A_2 + A_3,
    \end{equation}
    where
    \begin{subequations}
            \label{eq:sum_A4}
            \begin{align}
              \label{eq:sum_A0}
        A_0 &= \frac{\epsilon^2}{2\Delta t} \sum_j\Big[
        (I_{0,j}^{n+1}- I_{0,j}^{n})^2 + 3(I_{1,j}^{n} -
        I_{1,j}^{n-1})^2 + 5(I_{2,j}^{n} -
        I_{2,j}^{n-1})^2 \Big],\\
        \label{eq:sum_A1}
        A_1 &= -\frac{\epsilon}{2\Delta x}\sum_j  \Big[ I_{0,j}^{n+1}
        (I_{1, j+1}^{n} - I_{1, j-1}^{n}) +   I_{1,j}^n (I_{0, j+1}^{n} - I_{0,
          j-1}^{n}) + 2I_{1,j}^n (I_{2,
          j+1}^{n-1} - I_{2, j-1}^{n-1}) +  2I_{2,j}^n (I_{1, j+1}^{n} - I_{1,
          j-1}^{n})\Big], \\
        \label{eq:sum_A2}
        A_2 & = \frac{\alpha \epsilon}{2\Delta x}\sum_j  \Big[ I_{0,j}^{n+1}
        (I_{0, j+1}^{n} -  2I_{0, j}^{n}  +I_{0, j-1}^{n}) +  3I_{1,j}^n
        (I_{1, j+1}^{n-1} - 2 I_{1,j}^{n-1}  + I_{1, j-1}^{n-1})
        + 5I_{2,j}^n (I_{2,
          j+1}^{n-1} - 2I_{2, j}^{n-1} +  I_{2, j-1}^{n-1})\Big],\\
        A_3 &= -\sum_j \Big[-I_{0,j}^{n+1}(T_{j}^4)^{n+1} +
        (I_{0,j}^{n+1})^2 + 3(I_{1, j}^n)^2 + 5(I_{2,j}^n)^2\Big].
      \end{align}
    \end{subequations}
    Then we will begin from the approximation of $A_1$ and $A_2$. With
    some arrangement and the periodic boundary condition, $A_1$ is
    changed into
  \begin{equation}
    \label{eq:A1}
    \begin{aligned}
      A_1 = &-\frac{\epsilon}{2\Delta x}\sum_j \Big[ (I_{0,j}^{n+1} -
      I_{0,j}^{n}) (I_{1, j+1}^{n} - I_{1, j-1}^{n}) + 2(I_{1,j+1}^n -
      I_{1,j-1}^{n}) (I_{2, j}^{n} - I_{2, j}^{n-1}) \Big]\\
      \leqslant & \frac{\epsilon}{2\Delta x}\sum_j \Big[ \frac12\beta_1^2
      (I_{0,j}^{n+1} - I_{0,j}^{n})^2 + \frac{1}{2 \beta_1^2} (I_{1,
        j+1}^{n} - I_{1, j-1}^{n})^2 + \beta_2^2 (I_{2,j}^n -
      I_{2,j}^{n-1})^2 + \frac{1}{\beta_2^2} (I_{1, j+1}^{n} - I_{1, j-1}^{n})^2 \Big] \\
      \leqslant& \frac{\epsilon}{2\Delta x}\sum_j \Big[ \frac12\beta_1^2
      (I_{0,j}^{n+1} - I_{0,j}^{n})^2  + \beta_2^2 (I_{2,j}^n - I_{2,j}^{n-1})^2 +
      \left(\frac{2}{\beta_1^2}   + \frac{4}{\beta_2^2}\right) (I_{1, j}^{n})^2 \Big].
    \end{aligned}
    \end{equation}
    With the estimation that
    \begin{equation}
      \label{eq:local_A2_1}
      \begin{aligned}
        & \sum_j I_{0,j}^{n+1}(I_{0,j+1}^n - 2 I_{0,j}^n + I_{0,j-1}^n)\\
        & \qquad = \sum_j \Big[(I_{0,j}^{n+1} - I_{0,j}^n)(I_{0,j+1}^n - 2
        I_{0,j}^n + I_{0,j-1}^n) + I_{0,j}^n(I_{0,j+1}^n - 2 I_{0,j}^n
        +
        I_{0,j-1}^n)  \Big] \\
        &\qquad  \leqslant \sum_j\Big[\frac12\beta_3^2 (I_{0,j}^{n+1} - I_{0,j}^n)^2
        + \frac{2}{\beta_3^2}(I_{0,j}^n - I_{0,j-1}^n)^2\Big] - \sum_j
        (I_{0,j}^n - I_{0,j-1}^n)^2 \\
        &\qquad  = \sum_j\Big[\frac12\beta_3^2 (I_{0,j}^{n+1} - I_{0,j}^n)^2 +
        (\frac{2}{\beta_3^2} - 1)\left(I_{0,j}^n -
          I_{0,j-1}^n\right)^2\Big].
      \end{aligned}
    \end{equation}
    Similarly,  it also holds that
     \begin{equation}
      \label{eq:local_A2_2}
      \begin{aligned}
        \sum_j I_{1,j}^{n}(I_{1,j+1}^{n-1} - 2 I_{1,j}^{n-1} +
        I_{1,j-1}^{n-1}) &= \sum_j I_{1,j}^{n-1}(I_{1,j+1}^{n} - 2
        I_{1,j}^{n} +
        I_{1,j-1}^{n})  \\
        & \leqslant \sum_j\Big[\frac12\beta_4^2 (I_{1,j}^{n} -
        I_{1,j}^{n-1})^2 + 4\left(\frac{2}{\beta_4^2} -
          1\right)(I_{1,j}^n)^2\Big], \\
        \sum_j I_{2,j}^{n}(I_{2,j+1}^{n-1} - 2 I_{2,j}^{n-1} +
        I_{2,j-1}^{n-1}) &= \sum_j I_{2,j}^{n-1}(I_{2,j+1}^{n} - 2
        I_{2,j}^{n} +
        I_{2,j-1}^{n})  \\
        & \leqslant \sum_j\Big[\frac12\beta_5^2 (I_{2,j}^{n} -
        I_{2,j}^{n-1})^2 +
        \left(\frac{2}{\beta_5^2}-1\right)(I_{2,j}^n - I_{2,j-1}^n)^2\Big].
      \end{aligned}
    \end{equation}

    Let
    \begin{equation}
      \label{eq:app_coe}
      \beta_1^2 =  \frac{2 \epsilon  \Delta x}{\Delta t}-  2 \alpha,
      \qquad \beta_2^2 =  \frac{5\epsilon  \Delta x}{\Delta t}- 5
      \alpha, \qquad \beta_3^2 = 2, \qquad 
      \beta_4^2 = \frac{2 \Delta x \epsilon}{\alpha \Delta t},  \qquad
      \beta_5^2 = 2,
    \end{equation}
    then together with \eqref{eq:sum_A4}, \eqref{eq:A1},
    \eqref{eq:local_A2_1}, \eqref{eq:local_A2_2} and \eqref{eq:app_coe},
    \eqref{eq:sum_A3} is reduced into
      \begin{equation}
        \label{eq:sum_A5}
        \begin{split}
          \frac{\epsilon^2}{2} \sum_j\Big[ (I_{0,j}^{n+1})^2 -
          (I_{0,j}^{n})^2 + 3\left((I_{1,j}^{n})^2 - (I_{1,j}^{n-1})^2
          \right)+ 5 \left((I_{2,j}^{n})^2 -
            (I_{2,j}^{n-1})^2\right) \Big]\\
          \leqslant \beta_6\sum_j (I_{1,j}^{n})^2 + \Delta t \sum_j
          \Big[I_{0,j}^{n+1}(T_{j}^4)^{n+1} - (I_{0,j}^{n+1})^2 -
          5(I_{2,j}^{n})^2\Big],
                \end{split}
      \end{equation}
      with
      \begin{equation}
        \label{eq:coe_6}
        \beta_6 = \frac{9}{10}\dfrac{\epsilon \Delta t}{\Delta x}
        \left(\frac{1}{\frac{\epsilon \Delta x}{\Delta t} - \alpha}\right) + \dfrac{6 \alpha^2 (\Delta
        t)^2}{(\Delta x)^2}-\dfrac{6\alpha\epsilon\Delta t}{\Delta x} - 3 \Delta t.
    \end{equation}
    If it holds for $\beta_6$ that
      \begin{equation}
        \label{eq:coe_6_1}
        \beta_6 \leqslant 0, 
      \end{equation}
      with the time step length \eqref{eq:time}, then we can derive the
      stability result \eqref{eq:dis_energy}.  Precisely, with
      \eqref{eq:app_I0} and \eqref{eq:app_T}, we can derive that
      \begin{equation}
        \label{eq:T_1}
        \epsilon^2 \frac{T_j^{n+1} - T_j^n}{\Delta t}  = 
        -(T_j^4)^{n+1} + I_{0,j}^{n+1}.
 \end{equation}
 Multiplying \eqref{eq:T_1} with $(T_{j}^4)^{n+1}$ and summing over
 $j$, it holds with \eqref{eq:sum_A5}
  \begin{equation}
    \label{eq:sum_A6}
        \begin{split}
          & \sum_j \left[\frac{\epsilon^2}{2\Delta t}\Big(
            (I_{0,j}^{n+1})^2 - (I_{0,j}^{n})^2 +
            3\left[(I_{1,j}^{n})^2 - (I_{1,j}^{n-1})^2 \right]+ 5
            \left[(I_{2,j}^{n})^2 -
              (I_{2,j}^{n-1})^2\right]\Big) \right.\\
          &     \qquad \left. +\frac{\epsilon^2}{5 \Delta t} \Big
            [(T_j^5)^{n+1} -
            (T_j^5)^{n}\Big] \right] \leqslant - \sum_j
          \left[I_{0,j}^{n+1} - (T_{j}^4)^{n+1}\right]^2 \leqslant 0.
        \end{split}
      \end{equation}
      We derive the energy stability \eqref{eq:dis_energy}. The only
      point left is to prove \eqref{eq:coe_6_1}, which we will be done
      in two cases:
      \begin{enumerate}
      \item $\epsilon >  \Delta x$, in which case,
        \begin{equation}
          \label{eq:dt1}
          \Delta t= C \epsilon \Delta x. 
        \end{equation}
        Substituting \eqref{eq:dt1} into \eqref{eq:coe_6}, we can
        deduce that
        \begin{equation}
          \label{eq:coe_6_2}
          \beta_6 = \frac{9C^2\epsilon^2}{10}\left(\frac{1}{1 - \alpha C }\right)  + 6\alpha \epsilon^2C(\alpha C - 1)
          - 3C \Delta x \epsilon.
        \end{equation}
        Thus if
        \begin{equation}
          \label{eq:coe_6_3}
          0 < C < \min\left(\frac{\epsilon}{\alpha \Delta x},\frac{10 \Delta  x}{3 \epsilon + 10 \Delta x \alpha}
          \right),          
        \end{equation}
        it holds that $\beta_6 \leqslant 0$.
        
        For the coefficients $\beta_i^2, i = 1,\cdots 5$, it requires
        that $\beta_i^2 > 0$. Thus, from \eqref{eq:app_coe}, it
        demands that
        \begin{equation}
        \label{eq:coe_6_3_1}
            \frac{\epsilon \Delta x}{\Delta t} - \alpha > 0.
        \end{equation} 
        Substituting \eqref{eq:dt1} into \eqref{eq:coe_6_3_1}, we can obtain that 
        \begin{equation}
            \label{eq:coe_6_3_2}
            C < \frac{1}{\alpha}.
        \end{equation}
        Thus, the constrain on $C$ is changed into 
        \begin{equation}
        0 < C < \min\left(\frac{1}{\alpha},\frac{10 \Delta  x}{3 \epsilon + 10 \Delta x \alpha}
          \right).
    \end{equation}

      \item $\epsilon < \Delta x$, in which case
        \begin{equation}
          \label{eq:dt2}
          \Delta t= C \Delta x^2. 
        \end{equation}
        Substituting \eqref{eq:dt2} into \eqref{eq:coe_6}, we can
        deduce that
        \begin{equation}
          \label{eq:coe_6_4}
          \beta_6 =
          \frac{9}{10}C^2\epsilon \Delta x^2\left(\frac{1}{\epsilon - \alpha C \Delta x}\right)
          + 6\alpha  \Delta x C(\alpha C \Delta x - \epsilon)
          - 3C \Delta x^2.
        \end{equation}
        Thus if
        \begin{equation}
          \label{eq:coe_6_6}
          0 < C < \min\left(\frac{\epsilon}{\alpha \Delta x},\frac{10 \epsilon}{3 \epsilon + 10 \Delta x \alpha}
          \right),
        \end{equation}
        it holds that $\beta_6 \leqslant 0$.  Similarly, we can verify
        that the constrain $\beta_i^2 > 0, i = 1, \cdots 5$ will not
        affect the condition \eqref{eq:coe_6_6}, then the proof is
        finished.

        For $\epsilon < \Delta x$, it is always true that
        $\alpha = \exp(-1/\epsilon^2)$ is quite small, and
        \eqref{eq:coe_6_6} could be reduced into
        \begin{equation}
      \label{eq:coe_6_7}
      0 < C < \frac{10}{3}. 
    \end{equation}
      \end{enumerate}
    \end{proof}
}
    
\subsection{Analysis of the higher-order scheme}
\label{app:lossorder_proof}
From the test of the AP property for the numerical scheme, we found
that even for the IMEX3 scheme with WENO reconstruction, the
convergence order is only two. Analysis of the numerical scheme shows
that when solving $T^{n+1}$, the fourth-order polynomial equation of
$T^{n+1}$ is solved, where $(T^{n+1})^{4}$ is approximated as
\begin{equation}
 (T^{4})_{i} \approx (T_{i})^{4}
\end{equation}
instead of 
\begin{equation}
	(T)^{4}_{i} \approx \frac{\int_{x_{i-\frac{1}{2}}}^{x_{i + \frac{1}{2}}}T^{4}dx}{\Delta x},
\end{equation}
where $T_{i}$ is the cell average of cell $i$. Noting that 
\begin{equation}
	\frac{\int_{ x_{i - \frac{1}{2} }}^{ x_{i + \frac{1}{2} }} T(x) dx }{\Delta x} = T(x_{i}) + \frac{1}{24}(T(\xi_{i}))^{''}\Delta x^2,\quad \xi \in [x_{i - \frac{1}{2}}, x_{i + \frac{1}{2}}],
\end{equation}
and 
\begin{equation}
	\frac{	\int_{x_{i - \frac{1}{2}}}^{x_{i + \frac{1}{2}}} T^{4}(x) dx }{\Delta x} = T^{4}(x_{i}) + \frac{1}{24}(T^{4}(\eta_{i}))^{''}\Delta x^2,\eta_{i}\in[x_{i - \frac{1}{2}}, x_{i + \frac{1}{2}}],
\end{equation}
thus, it holds 
\begin{equation}
		\frac{\int_{x_{i - \frac{1}{2}}}^{x_{i + \frac{1}{2}}} T^{4} dx}{\Delta x} - (T_{i})^{4} = \mathcal{O}(\Delta x^{2}).
	\end{equation}
Therefore, the convergence order of the whole numerical scheme is at most two.

\bibliographystyle{plain}
\bibliography{reference}
\end{document}